\documentclass{amsart}
\usepackage{amssymb}
\usepackage{mathrsfs}
\usepackage{stmaryrd}
\usepackage{bbm}
\usepackage{oldgerm}
\usepackage[english]{babel}
\usepackage[T1]{fontenc}
\usepackage[latin1]{inputenc}
\usepackage[all]{xy}
\usepackage{hyperref}

\newtheorem{theorem}[subsection]{Theorem}
\newtheorem{proposition}[subsection]{Proposition}
\newtheorem{corollary}[subsection]{Corollary}
\newtheorem{lemma}[subsection]{Lemma}

\theoremstyle{definition}

\newtheorem{remark}[subsection]{Remark}

\numberwithin{equation}{subsection}

\newlength{\dhatheight}
\newcommand{\swdhat}[1]{
    \settoheight{\dhatheight}{\ensuremath{\hat{#1}}}
    \addtolength{\dhatheight}{-0.70ex}
    \widehat{\vphantom{\rule{1pt}{\dhatheight}}
    \smash{\widehat{#1}}}}
\newcommand{\wdhat}[1]{
    \settoheight{\dhatheight}{\ensuremath{\widehat{#1}}}
    \addtolength{\dhatheight}{-0.35ex}
    \widehat{\vphantom{\rule{1pt}{\dhatheight}}
    \smash{\widehat{#1}}}}

\newcommand{\m}{\mathfrak m}
\newcommand{\bp}{\bigoplus}
\newcommand{\bt}{\bigotimes}
\newcommand{\ot}{\otimes}
\newcommand{\ti}{\times}
\newcommand{\op}{\oplus}

\newcommand{\ol}{\overline}

\newcommand{\gal}{\mathrm{Gal}}
\newcommand{\csk}{\m_K/\m_K^2}
\newcommand{\csl}{\m_L/\m_L^2}
\newcommand{\ra}{\rightarrow}

\newcommand{\xra}{\xrightarrow}

\newcommand{\lra}{\longrightarrow}

\newcommand{\ira}{\hookrightarrow}

\newcommand{\sra}{\twoheadrightarrow}

\newcommand{\ma}{\mapsto}

\newcommand{\sube}{\subseteq}
\newcommand{\sub}{\subset}
\newcommand{\dr}{\mathrm d}
\newcommand{\skl}{S_{K,L}}
\newcommand{\slk}{S_{L/K}}

\newcommand{\f}{\frac}
\newcommand{\lt}{\left}
\newcommand{\rt}{\right}
\newcommand{\OK}{\mathcal{O}_K}
\newcommand{\OL}{\mathcal{O}_L}
\newcommand{\cO}{\mathcal{O}}
\newcommand{\isora}{\xra{\sim}}

\newcommand{\gk}{G_K}
\newcommand{\gkl}{G_{K,\log}}
\newcommand{\red}{\mathrm{red}}
\newcommand{\grl}{\mathrm{Gr}^r_{\log}{G_K}}

\newcommand{\thlogr}{\Th^{(r)}_{\ol F,\log}}

\newcommand{\thlogc}{\Th^{(c)}_{\ol F,\log}}
\newcommand{\thnonc}{\Th^{(c)}_{\ol F}}

\newcommand{\thnonr}{\Th^{(r)}_{\ol F}}
\newcommand{\xic}{\Xi^{(c)}_{\ol F}}

\newcommand{\xir}{\Xi^{(r)}_{\ol F}}

\newcommand{\ab}{\mathrm{ab}}
\newcommand{\alg}{\mathrm{alg}}

\newcommand{\plim}{\varprojlim}

\newcommand{\bff}{\mathbf{f}}
\newcommand{\cA}{\mathcal{A}}
\newcommand{\what}{\widehat}
\newcommand{\fek}{\text{F\'E}_{/K}}

\newcommand{\cR}{\mathcal{R}}
\newcommand{\afs}{\mathrm{AFS}_{\cO_{K}}}
\newcommand{\bmg}{\mathbbm{g}}
\newcommand{\bfg}{\mathbf{g}}

\newcommand{\id}{\mathrm{id}}
\newcommand{\beq}{\begin{equation}}
\newcommand{\eeq}{\end{equation}}
\newcommand{\kcc}{\mathrm{KCC}}
\newcommand{\cc}{\mathrm{CC}}
\DeclareMathOperator{\gr}{Gr}
\DeclareMathOperator{\spec}{Spec}
\DeclareMathOperator{\Sp}{Sp}
\DeclareMathOperator{\Tr}{Tr}
\DeclareMathOperator{\Hom}{Hom}
\DeclareMathOperator{\rsw}{rsw}

\DeclareMathOperator{\sw}{sw}
\DeclareMathOperator{\ind}{Ind}
\DeclareMathOperator{\res}{Res}
\DeclareMathOperator{\ord}{ord}
\DeclareMathOperator{\tr}{tr}
\DeclareMathOperator{\aut}{Aut}

\DeclareMathOperator{\spf}{Spf}
\DeclareMathOperator{\sym}{Sym}
\DeclareMathOperator{\dimtot}{dimtot}
\DeclareMathOperator{\rank}{rank}

\def\wt{\widetilde}
\def\bmf{\mathbbm{f}}
\def\sC{\mathscr{C}}

\def\fp{\mathfrak{p}}
\def\fX{\mathfrak{X}}
\def\sO{\mathscr{O}}

\def\sF{\mathscr{F}}
\def\^{\wedge}
\def\hO{\what {\Omega}^1}

\def\le{\leqslant}
\def\di{\mathfrak{D}}
\def\A{\mathbb{A}}
\def\Z{\mathbb Z}
\def\F{\mathbb F}
\def\Q{\mathbb Q}
\def\<{\langle}
\def\>{\rangle}
\def\a{\alpha}
\def\b{\beta}
\def\c{\chi}
\def\g{\gamma}

\def\d{\delta}
\def\D{\Delta}
\def\l{\lambda}
\def\L{\Lambda}

\def\ve{\varepsilon}

\def\o{\omega}
\def\O{\Omega}
\def\p{\pi}

\def\r{\rho}
\def\s{\sigma}
\def\t{\tau}
\def\th{\theta}
\def\Th{\Theta}
\def\k{\kappa}
\def\l{\lambda}

\def\v{\vartheta}
\def\x{\xi}

\begin{document}

\title{Ramification and nearby cycles for $\ell$-adic sheaves on relative curves}
\author{Haoyu Hu}
\address{Haoyu Hu, IHES, Le Brois-Marie, 35 Rue de Chartres, 91440 Bures-sur-Yvette, France.}

\email{ haoyu@ihes.fr \& huhaoyu@mail.nankai.edu.cn}
\begin{abstract}
Deligne and Kato proved a formula computing the dimension of the nearby cycles complex of an $\ell$-adic sheaf on a relative curve over an excellent strictly henselian trait. In this article, we reprove this formula using Abbes-Saito's ramification theory.
\end{abstract}

\maketitle

\tableofcontents

\section{Introduction}
\subsection{}\label{set intro}
Let $R$ be an excellent strictly henselian discrete valuation ring of residue characteristic $p>0$, $S=\spec(R)$, $s$ (resp. $\eta$, resp. $\bar\eta$) the closed point (resp. the generic point, resp. a geometric generic point) of $S$. Let $\mathfrak{X}$ be a smooth relative curve over $S$, $x$ a closed point of the special fiber $\fX_s$, $X$ the strict henselization of $\fX$ at $x$, $U$ a non-empty open sub-scheme of $X_{\eta}$, and $u:U\ra X_{\eta}$ the canonical injection. Let $\L$ a finite field of characteristic $\ell\neq p$, and $\sF$ a locally constant constructible \'etale sheaf of $\L$-module on $U$. The spaces of nearby cycles of $\sF$
\begin{equation*}
\Psi^i_x(u_!\sF)=\mathrm{H}^i_{\text{\'et}}(X_{\bar\eta},u_!\sF)\ \ \ (i\geqslant 0)
\end{equation*}
vanish when $i\geqslant 2$ (\cite{sga7ii} XIII, \cite{fu} 9.2.2) and the dimension of $\Psi^0_x(u_!\sF)$ is easy to compute. The aim of this article is to reprove a Deligne-Kato's formula that computes the dimension of $\Psi_x^1(u_!\sF)$ \cite{lau,kato vc,kato scdv} using Abbes-Saito's ramification theory \cite{as i,as ii}.

\subsection{}
Let $\fp$ be the generic point of the special fiber $X_s$. We denote by $\k(\fp)$ the residue field of $\fp$, which is the fraction field of a strictly henselian discrete valuation ring. Assume first that $\sF$ can be extended to a locally constant constructible sheaf
$\wt \sF$ on an open sub-scheme $\wt U$ of $X$ containing $\fp$. Then Deligne computes the dimension of $\Psi^1_x(u_!\sF)$. Let $\sw_{\fp}(\wt{\sF})$ be the Swan conductor of the pull-back of $\wt\sF$ on $\spec(\k(\fp))$ and let
\begin{equation*}
\varphi(s)=\sw_{\fp}(\wt\sF)+\rank(\sF).
\end{equation*}
On the other hand, for any $t\in X_{\bar\eta}-U_{\bar\eta}$, let $\sw_{t}(\sF)$ be the Swan conductor of the pull-back of $\sF$ on $\spec(\cO_{X_{\bar\eta},t})\ti_XU$, and let
\begin{equation*}
\varphi(\eta)=\sum_{t\in X_{\bar\eta}-U_{\bar\eta}}(\sw_t(\sF)+\rank(\sF)).
\end{equation*}
Then, Deligne's formula is (\cite{lau} 5.1.1)
\begin{equation}\label{deligne}
\dim_{\L}\Psi^0_x(u_!\sF)-\dim_{\L}\Psi^1_x(u_!\sF)=\varphi(s)-\varphi(\eta).
\end{equation}
\subsection{}\label{kato mtds}
Kato generalized Deligne's formula for any $\sF$. His formula has the same form as \eqref{deligne}. The definition of the invariant $\varphi(\eta)$ is the same as above, but $\varphi(s)$ cannot be defined by the same method. He provided two definitions of $\varphi(s)$. The first one uses a ramification theory for valuation rings of rank two, that he developed for this purpose \cite{kato vc}. The second one uses his notion of Swan conductors with differential values \cite{kato scdv}. Both methods rely on Epp's partial semi-stable reduction theorem \cite{epp}. In this article, we define the invariant $\varphi(s)$ in terms of ramification theory of Abbes and Saito \cite{as i,as ii}. The case when $\sF$ has rank $1$ is due to Abbes and Saito (\cite{as ft} Appendix A).

\subsection{}
Let $K$ be a complete discrete valuation field, $\OK$ its integer ring, $\m_K$ the maximal ideal of $\OK$ and $F$ the residue field of $\OK$. We assume that $F$ is of finite type over a perfect field $F_0$ of characteristic $p$. We denote by $\ol K$ a separable closure of $K$, by $\cO_{\ol K}$ the integral closure of $\OK$ in $\ol K$, by $\ol F$ the residue field of $\cO_{\ol K}$, by $v$ the valuation of $\ol K$ normalized by $v(K^{\ti})=\Z$ and by $G_K$ the Galois group of $\ol K/K$. Abbes and Saito defined a decreasing filtration $G^r_{K,\log}$ ($r\in \Q_{\geqslant 0}$) of $G_K$, called the logarithmic ramification filtration. For any rational number $r\geqslant0$, we put $G_{K,\log}^{r+}=\ol{\bigcup_{b>r}G^b_{K,\log}}$. Then $P=G_{K,\log}^{0+}$ is the wild inertia subgroup of $G_K$ (\cite{as i} 3.15). For any rational number $r>0$, the graded piece
\begin{equation*}
\gr^r_{\log}G_K=G_{K,\log}^r\big/G^{r+}_{K,\log}
\end{equation*}
is abelian and killed by $p$ (\cite{saito cc} 1.24, \cite{as iii} Th. 2).

For any $r\in\Q$, we denote by $\m^r_{\ol K}$ (resp. $\m^{r+}_{\ol K}$) the set of elements of $\ol K$ such that $v(x)\geqslant r$ (resp. $v(x)>r$). Let $\O^1_F(\log)$ be the $F$-vector space
\begin{equation*}
\O^1_F(\log)=(\O^1_{F/F_0}\op(F\ot_{\Z} K^{\ti}))/(\dr \bar a-\bar a\ot a;\;a\in\OK^{\ti}),
\end{equation*}
where $\bar a$ is the residue class of $a$ in $F$. We have a canonical exact sequence of finite dimensional $F$-vector spaces
\begin{equation*}
0\ra\O^1_F\ra \O^1_F(\log)\ra F\ra 0.
\end{equation*}
For any rational number $r>0$, there exists a canonical injective homomorphism (\cite{saito cc} 1.24, \cite{as iii} Th. 2), called the {\em refined Swan conductor},
\begin{equation*}
\rsw:\Hom_{\F_p}(\gr^r_{\log}G_K,\F_p)\ra \O^1_F(\log)\ot_F\m^{-r}_{\ol K}/\m^{-r+}_{\ol K}.
\end{equation*}

Let $M$ be a finite dimensional ${\L}$-vector space on which $P$ acts through a finite discrete quotient,
\begin{equation*}
M=\op_{r\in\Q_{\geqslant 0}} M^{(r)}
\end{equation*}
the slope decomposition of $M$ (cf. \ref{slope decom lemma}), and for any rational number $r>0$,
\begin{equation*}
M^{(r)}=\op_{\chi}M^{(r)}_{\chi}
\end{equation*}
the central character decomposition of $M^{(r)}$, where the sum runs over finitely many characters $\chi:\gr^r_{\log}G_K\ra {\L}_{\chi}^{\ti}$ such that ${\L}_{\chi}$ is a finite extension of ${\L}$ (cf. \ref{center char decomp}). Enlarging ${\L}$, we may assume that for all rational number $r>0$ and for all central characters $\chi$ of $M^{(r)}$, ${\L}={\L}_{\chi}$. We fix a non-trivial character $\psi_0:\F_p\ra {\L}^{\ti}$. Since $\gr^r_{\log}G_K$ is abelian and killed by $p$, $\chi$ factors uniquely through $\gr^r_{\log}G_K\ra\F_p\xra {\psi_0} {\L}^{\ti}$. We denote abusively by $\c:\gr^r_{\log}G_K\ra\F_p$ the induced character. We fix a uniformizer $\p$ of $\OK$. We define {\em Abbes-Saito's characteristic cycle} of $M$ and denote by $\cc_{\psi_0}(M)$ the following section \eqref{cc formula}
\begin{equation*}
\cc_{\psi_0}(M)=\bt_{r\in\Q_{> 0}}\bt_{\c\in X(r)}(\rsw(\c)\ot\p^{r})^{\dim_{\L} M^{(r)}_{\c}}\in (\O^1_F(\log)\ot_{F}\ol F)^{\ot \dim_{A}M/M^{(0)}}.
\end{equation*}

\subsection{}\label{cc=kcc intro}
In the following, we assume that $p$ is not a uniformizer of $K$ (i.e. either $K$ has characteristic $p$ or $K$ has characteristic zero and $p$ is not a uniformizer of $\OK$). Let $L$ be a finite Galois extension of $K$ of Group $G$. We assume that $L/K$ has ramification index one and that the residue field extension is non-trivial, purely inseparable and monogenic~; we say that the extension $L/K$ is of type (II) (c.f. \ref{type}). Let $M$ be a finite ${\L}$-vector space on which $G_K$ acts through $G$. We prove that, for any rational number $r>0$, and any central character $\chi:\gr^r_{\log}G_K\ra \F_p$ of $M^{(r)}$, we have (\ref{thlog factor xi})
\begin{equation*}
\rsw(\c)\in \O^1_F\ot_{F}\m^{-r}_{\ol K}/\m^{-r+}_{\ol K}.
\end{equation*}
Hence, we have $\cc_{\psi_0}(M)\in (\O^1_F\ot_F\ol F)^{\ot m}$, where $m=\dim_AM/M^{(0)}$ (\ref{cc in O}). On the other hand, using Kato's theory of Swan conductors with differential values, we can define {\em Kato's characteristic cycle} $\kcc_{\psi_0(1)}(M)$ \eqref{kcc formula}. Our main result \eqref{general equal} is the following equality
\begin{equation}\label{general equal intro}
\cc_{\psi_0}(M)=\kcc_{\psi_0(1)}(M).
\end{equation}
Using Kato's theory, we deduce a Hasse-Arf type theorem (\ref{hasse arf cc})
\begin{equation*}
\cc_{\psi_0}(M)\in (\O^1_F)^m\sub(\O^1_F\ot_F\ol F)^m,
\end{equation*}
and an induction formula \eqref{ind for cc} for Abbes-Saito's characteristic cycle.
\subsection{}
Under the assumptions of \ref{set intro}, we can now give the new definition of $\varphi(s)$. Firstly, by Epp's results \cite{epp}, we can reduce to the case where $\sF$ is trivialized by a Galois \'etale connected covering $U'$ of $U$ such that the special fiber of the normalization $X'$ of $X$ in $U'$ is reduced. We denote by $\what\cO_{X,\fp}$ the completion of $\cO_{X,\fp}$, by $K_{\fp}$ the fraction field of $\what\cO_{X,\fp}$ and by $\sF_{\fp}$ the representation of $\gal(K^{\mathrm{sep}}_{\fp}/K_{\fp})$ corresponding to the pull-back of $\sF$ on $\spec(\what\cO_{X,\fp})\ti_XU$. The latter factors through the Galois group of a finite Galois extension $L_{\fp}$ of $K_{\fp}$, which is of type (II) over an unramified extension of $K_{\fp}$. We fix a uniformizer $\p$ of $R$ and a non-trivial character $\psi_0:\F_p\ra\L^{\ti}$. We still have $\cc_{\psi_0}(\sF_{\fp})\in (\O^1_{\k(\fp)})^{\ot m}$ (cf. \ref{remark general equal}). We denote by $\ord_{\fp}$ the valuation of $\kappa(\fp)$ normalized by $\ord_{\fp}(\kappa(\fp)^{\ti})=\Z$ and abusively by $\ord_{\fp}:\O^1_{\kappa(\fp)}-\{0\}\ra \Z$ the map defined by $\ord_{\fp}(\a\dr\b)=\ord_{\fp}(\a)$, if $\a,\b\in\kappa(\fp)^{\ti}$ and $\ord_{\fp}(\b)=1$. The latter can be uniquely extended to $(\O^1_{\k(\fp)})^{\ot r}-\{0\}$ for any integer $r\geqslant 1$. We denote by $\ol{\sF}_{\fp}$ the restriction to $\spec(\kappa(\fp))$ of the direct image of $\sF_{\fp}$ by the map $\spec (K_{\fp})\ra \spec(\what {\cO}_{X,\fp})$. It  corresponds to a representation of $\gal(\ol{\k(\fp)}/\k(\fp))$. The invariant $\varphi(s)$ is defined by
\begin{equation}\label{phi s}
\varphi(s)=-\ord_{\fp}(\cc_{\psi_0}(\sF_{\fp}))+\sw_s(\ol{\sF}_{\fp})+\rank(\ol{\sF}_{\fp}).
\end{equation}
In fact, Kato's second definition of $\varphi(s)$ (\cite{kato scdv} 4.4) is obtained by replacing $\cc_{\psi_0}(\sF_{\fp})$ by $\kcc_{\psi_0(1)}(\sF_{\fp})$ in \eqref{phi s}. Hence, from \eqref{general equal intro}, we deduce that Deligne-Kato's formula \eqref{deligne} holds true with our definition (cf. \ref{deligne kato}).

\subsection{}
Deligne-Kato's formula has already had important applications. For instance, Deligne's formula could be used in Laumon's work on local Fourier transform (\cite{lautf} 2.4.3) and Kato's formula was recently used in the work of Obus and Wewers on local lifting problem \cite{ow}. We would like to mention that Laumon's formula of the rank of the local Fourier transform is a direct application of the formulation of Deligne-Kato's formula using \eqref{phi s}. Indeed, it was reproved in (\cite{as ft} Appendix B) by reducing to the rank 1 case by Brauer theorem.

\subsection{}
This article is organized as follows. We briefly introduce Kato's swan conductor with differential values and Abbes-Saito's ramification theory in $\S$3 and $\S$4, respectively. We study in $\S$5 the ramification of extensions of type (II). We recall tubular neighborhoods and normalized integral models in $\S$6. We study the isogeny associated to an extension of type (II) in $\S$7 in the equal character case and in $\S$8 in the unequal characteristic case. Using the results of these two sections, we prove the main theorem \ref{theorem rsw} in $\S$9. In $\S$10, the heart of this article, we compare Kato's characteristic cycle and Abbes-Saito's characteristic cycle. The last section is devoted to Deligne-Kato's formula by using Abbes-Saito's characteristic cycle.

\subsection*{Acknowledgement}
This article is a part of the author's thesis at Universit\'e Paris-Sud and Nankai University. The author would like to express his deepest gratitude to his supervisors Ahmed Abbes and Lei Fu for leading him to this area and for patiently guiding him in solving this problem. The author is also grateful to Fonds Chern and Fondation Math\'ematiques Jacques Hadamard for their support during his stay in France.

\section{Notation}
\subsection{}\label{basic notes}
In this article, $K$ denotes a complete discrete valuation field, $\OK$ its integer ring, $\m_K$ the maximal ideal of $\OK$ and $F$ the residue field of $\OK$. We assume that the characteristic of $F$ is $p>0$. We fix a uniformizer $\p$ of $\OK$. Let $\ol K$ be a separable closure of $K$, $G_K$ the Galois group of $\ol K$ over $K$, $\cO_{\ol K}$  the integral closure of $\OK$ in $\ol K$, $\ol F$ the residue field of $\cO_{\ol K}$ and $v$ the valuation of $\ol K$ normalized by $v(K^{\ti})=\Z$.  We denote by $\fek$ the category of finite \'etale $K$-algebras. For any object $K'$ of $\fek$, we denote by $\cO_{K'}$ the integer ring of $K'$ and by $\m_{K'}$ the radical of $\cO_{K'}$.

\subsection{}\label{space}
For a field $k$ and one dimensional $k$-vector spaces $V_1,\dots, V_m$, we denote by $k\< V_1,\dots,V_m \>$ the $k$-algebra
\begin{equation*}
\label{plus} \bp\limits_{(i_1,...,i_m)\in {\Z}^m} V_1^{\ot i_1}\ot \cdots \ot V_m^{\ot i_m},
\end{equation*}
and by $(k\<V_1,\dots,V_m\>)^{\times}$ its group of units. An element of $(k\<V_1,\dots,V_m\>)^{\ti}$ is contained in some vector space $V_1^{\ot i_1}\ot \cdots \ot V_m^{\ot i_m}$. Such an element $x$ will be denoted by $[x]$ and we adopt the additive notation, i.e. $[x]+[y]=[x\cdot y]$ and $-[x]=[x^{-1}]$. If for each $1\le i\le m$, $e_i$ is a non-zero element of $V_i$, we have an isomorphism
\begin{equation*}
 k\<V_1,\dots,V_m\>\isora k[X_1,\dots,X_m,X_1^{-1},\dots,X_m^{-1}],\ \ \
e_i\ma X_i,
\end{equation*}
and hence an isomorphism
\begin{equation} \label{conductor space}
(k\<V_1,\dots,V_m\>)^{\times}\isora k^{\times}\op\Z^m.
\end{equation}

\section{Kato's Swan conductors with differential values}
\subsection{}
In this section, we fix a finite separable extension $L$ of $K$ of ramification index $e$ contained in $\ol K$. We denote by $\OL$ its integer ring and by $E$ the residue field of $\OL$.

\subsection{}\label{rkrl}
We denote the group $(F\<\csk\>)^{\ti}$ by $R_K$ and the group $(E\<\csl\>)^{\ti}$ by $R_L$ (cf. \ref{space}).
The canonical isomorphisms
\begin{equation}\label{iso cs1}
E\ot_F(\csk)\isora \m_{L}^e/\m_{L}^{e+1},
\end{equation}
\begin{equation}\label{iso cs2}
(\csl)^{\ot e}\isora\m_{L}^e/\m_{L}^{e+1},
\end{equation}
induce an injective homomorphism of $F$-algebras
\begin{equation*}
     F\<\csk\>\ra E\<\csl\>
\end{equation*}
and hence an injective homomorphism $R_K\ra R_{L}$.

\subsection{}\label{type}
Kato's theory applies if the extension $L/K$ is of one of the following types (\cite{kato scdv}, 1.5):
\begin{itemize}
\item[(I)]  $L/K$ is totally ramified (i.e. $F=E$)~;
 \item[(II)] the ramification index of $L/K$ is $1$ and the residue field extension $E/F$ is purely inseparable and monogenic.
\end{itemize}
Observe that in both cases, $\OL$ is monogenic over $\OK$. These two cases do not cover all finite separable extensions.

In the remaining part of this section, we assume that $L/K$ is of type (II). We denote by $p^n$ the degree of the residue extension $E/F$. We choose an element $h\in\OL$ such that its reduction $\bar h\in E$ is the generator of $E/F$ and a lifting $a\in \OK$ of  $\bar a=\bar h^{p^n}\in F$.

\begin{lemma}\label{frob}
Let $V$ be the kernel of the canonical morphism $\O^1_F\ra\O^1_E$. Denote by $\varrho$ the morphism $E\ra F,\;b\ma b^{p^n}$, by $\phi$ the morphism $F\ra F,\;b\ma b^{p^n}$, and by $\varphi$ the morphism $E\ra E,\;b\ma b^{p^n}$.
\begin{itemize}
\item[(i)]
The $F$-vector space $V$ is of dimension $1$, generated by $\dr \bar a$.
\item[(ii)]
The $E$-vector space $\O^1_{E/F}$ is of dimension $1$, generated by $\dr \bar h$.
\item[(iii)]
The canonical morphism $F\ot_{\varrho,E}\O^1_{E/F}\ra \O^1_{F/\phi(F)}=\O^1_F$ associated to $F\ra E\xra{\varrho} F$
is injective with image $V$.
\item[(iv)]
For any $1$-dimensional $E$ vector space $W$, the morphism
\begin{equation*}
E\ot_{\varphi,E}W\ra W^{\ot p^n}, \ \ \ y\ot z  \ma yz^{\ot{p^n}}
\end{equation*}
is an isomorphism.
\item[(v)]
There exist a canonical $E$-linear isomorphism
\begin{equation}\label{vo}
E\ot_F V\isora (\O^1_{E/F})^{\ot p^n},
\end{equation}
that maps $y\ot\dr \bar a$ to $y(\dr \bar h)^{\ot p^n}$.
\end{itemize}
\end{lemma}

\begin{proof}
(i), (ii), (iv) are obvious.
We have two canonical exact sequences of differential modules corresponding to the extensions $\phi:F\ra E\xra{\varrho} F$ and $\varphi:E\xra{\varrho}F\ra E$,
\begin{equation*}
F\ot_{\varrho,E}\O^1_{E/F}\xra{\b}\O^1_F\ra\O^1_{F/\varrho(E)}\ra0,
\end{equation*}
\begin{equation*}
E\ot_{F}\O^1_{F/\varrho(E)}\ra\O^1_E\ra\O^1_{E/F}\ra 0.
\end{equation*}
Since the canonical morphism $\O^1_F\ra\O^1_E$ factors as
\begin{equation*}
\O^1_F\ra\O^1_{F/\varrho(E)}\ra E\ot_{F}\O^1_{F/\varrho(E)}\ra\O^1_E,
\end{equation*}
the image of $F\ot_{\varrho,E}\O^1_{E/F}$ in $\O^1_E$ is $\{0\}$. Hence the image of $\b$ lies in $V$. Since the kernel of $\O^1_F\ra\O^1_{F/\varrho(E)}$ is not zero (as it contains $\dr \bar a$) and since $F\ot_{\varrho,E}\O^1_{E/F}$ is of dimension 1, $\b$ is injective. Hence $\b$ induces an isomorphism
\begin{equation*}
\b:F\ot_{\varrho,E}\O^1_{E/F}\isora V.
\end{equation*}
From (ii) and (iv), we obtain an isomorphism
\begin{equation*}
\b':E\ot_{\varphi,E}\O^1_{E/F}\ra(\O^1_{E/F})^{\ot p^n}, \ \ \ y\ot z \dr \bar h\ma yz^{p^n}(\dr \bar h)^{\ot p^n}.
\end{equation*}
We take for \eqref{vo} the isomorphism $\b'\circ(\id_E\ot\b)^{-1}$.
\end{proof}

\subsection{} \label{2type}
Let $V$ be the kernel of the canonical morphism $\O^1_F\ra\O^1_E$ (\ref{frob}). We put \eqref{space}
\begin{equation*}
\skl=(F\<\csk, V\>)^{\times} \quad \mathrm{and} \quad\slk=(E\<\csl, \O_{E/F}^1\>)^{\times}.
\end{equation*}
From \eqref{iso cs1} and \eqref{vo}, we obtain an injective homomorphism of $F$-algebras
\begin{equation*}
 F\<\csk,V\>\ira E\<\csl,\O^1_{E/F}\>,
\end{equation*}
which induces an injective homomorphism
\begin{equation}
\label{injs} \skl\ira\slk.
\end{equation}

\subsection{}Let $L'$ be a subfield of $L$ containing $K$, $\cO_{L'}$ its integer ring and $E'$ its residue field.
When $L'\neq L$ (resp. $L'\neq K$), the extension $L/L'$ (resp. $L'/K$) is of type (II)~; we consider $S_{L',L}$ (resp. $S_{L'/K}$) as a subgroup of $S_{L/K}$ containing $\skl$, by functoriality. If $K\neq L'\neq L$, the following canonical maps
\begin{equation*}
\ker(\O^1_F\ra\O^1_{E'})\ra\ker(\O^1_F\ra\O^1_E),
\end{equation*}
\begin{equation*}
\O^1_{E/F}\ra\O^1_{E/E'},
\end{equation*}
\begin{equation*}
\ker(\O^1_{E'}\ra\O^1_E)\ra\O^1_{E'/F}
\end{equation*}
are isomorphisms by considering dimensions, which give the following relations:
\begin{equation*}
\skl=S_{K,L'}\sub S_{L'/K}=S_{L',L}\sub S_{L/L'}=S_{L/K}.
\end{equation*}

\subsection{}\label{different}
Let $i$ be the maximal integer such that $\Tr_{L/K} (m^i_L)=O_K$. The surjective homomorphism $\Tr_{L/K}:\m^i_L/\m^{i+1}_L\ra O_K/\m_K=F$ induces an $E$-isomorphism
\begin{equation*}
\m^i_L/\m^{i+1}_L\isora \Hom_F(E,F),\ \ \ b\ma (a\ma\Tr_{L/K}(ab)),
\end{equation*}
and hence a basis of $(\csl)^{\ot(-i)}\ot_E\Hom_F(E,F)$, that we call Kato's different of $L/K$ and denote by $\di(L/K)$ (\cite{kato scdv} 2.1).

\subsection{}
Following Kato (\cite{kato scdv} 2.3), there is an $F$-linear map  $\Tr_{E/F}:\O_E^1\ra\O^1_F$ characterized by
\begin{equation*}
\Tr_{E/F}\lt(\f{\dr x}{x}\rt)=\f{\dr x^{p^n}}{x^{p^n}},\quad \Tr_{E/F}\lt(x^i\f{\dr x}{x}\rt)=0,
\end{equation*}
for any $x\in E^{\ti}$ and $1\le i\le p^n-1$. Its image is $V$ (\ref{frob}) and it induces an isomorphism
\begin{equation}\label{ohom}
\O^1_{E/F}\xra{\sim}\Hom_F(E,V),\quad \o\ma(a\ma \Tr_{E/F}(a\o)).
\end{equation}
Hence we obtain a sequence of isomorphisms
\begin{equation}\label{r to s}
\Hom_F(E,F)\xra{\eqref{ohom}}\O^1_{E/F}\ot_F V^{\ot (-1)}\xra{\eqref{vo}}\O^1_{E/F}\ot_E(\O^1_{E/F})^{\ot(-p^n)}=(\O^1_{E/F})^{\ot(1-p^n)},
\end{equation}
by which $E\<\csl\>\ot_E\Hom_F(E,F)$ is a sub-$E\<\csl\>$-module of $E\<\csl,\O^1_{E/F}\>$. Hence we may consider $\di(L/K)$ (\ref{different}) as an element of $S_{L/K}$.

\begin{proposition}[\cite{kato scdv} 2.2]
Let $L'$ be a subfield of $L$ containing $K$. If $L=L'$ (resp. $L'=K$), we put $\di(L/L')=[1]$ (resp. $\di(L'/K)=[1]$). Then, we have
\begin{equation}\label{relation diff}
\di(L/K)=\di(L/L')+\di(L'/K)\in\slk.
\end{equation}
\end{proposition}
We consider $\di(L'/K)\in S_{L'/K}\sube\slk$.

\subsection{}\label{def s}
In the rest of this section, we assume that the extension $L/K$ is Galois of group $G$. For any $\s\in G-\{1\}$, we put
\begin{equation*}
s_G(\s)=[\dr\bar h]-[h-\s(h)]\in\slk,
\end{equation*}
where the term $[\dr\bar h]$ corresponds to the element $\dr \bar h$ in $\O^1_{E/F}$ and the term $[h-\s(\s)]$ corresponds abusively to the class of $h-\s(h)\in (\m_L/\m_L^2)^{\ot v(h-\s(h))}$. The definition of $s_G(\s)$ is independent of the choice of the generator $h$ (\cite{kato scdv} 1.8). We also put
\begin{equation}\label{s1}
s_G(1)=-\sum_{\s\in G-\{1\}}s_G(\s)\in\slk.
\end{equation}
We have (\cite{kato scdv} (2.4))
\begin{equation}\label{sg1}
s_G(1)=\di(L/K).
\end{equation}

\begin{proposition}[\cite{kato scdv} Prop. 1.9]
Let $H$ be a normal subgroup of $G$. Then for any element $\t\in G/H-\{1\}$, we have
\begin{equation*}
s_{G/H}(\t)=\sum_{\substack{\s\in G\\\s\ma\t}}s_G(\s).
\end{equation*}

\end{proposition}

\subsection{}\label{<c,1>}
In the following of this section, let $C$ be an algebraically closed field of characteristic zero, $\x$ a primitive $p$-th root of $1$ in $C$ and $\widetilde{\Z}$ the integral closure of $\Z$ in $C$. For any finite group $H$, we denote by $R_C(H)$ the Grothendieck group of finitely generated $C[H]$-modules. For an element $\c\in R(H)$, let $\<\c, 1\>=\frac{1}{\sharp H}\sum_{\s\in H}\tr_{\c}(\s)$.

\subsection{}\label{def swan}
For an element $\c\in R_C(G)$, we put
\begin{eqnarray}
s_G(\c)&=&\sum_{\s\in G} s_G(\s)\ot\tr_{\c}(\s)\in\slk\ot_{\Z}\widetilde{\Z},\nonumber\\
\ve(\x)&=&\sum_{r\in\F_p^{\ti}\sube E^{\ti}}[r]\ot\x^r\in \slk\ot_{\Z}\widetilde{\Z}.\nonumber
\end{eqnarray}
Kato defined the {\em Swan conductor with differential values} of $\c$ as
\begin{equation}\label{swan}
\sw_{\x}(\c)=s_G(\c)+(\dim\c-\<\c,1\>)\ve(\x)\in\slk\ot\widetilde{\Z}.
\end{equation}
For any $r\in \F^{\ti}_p$, we have $\sw_{\x^r}(\c)=\sw_{\x}(\c)+(\dim{\c}-\<\c,1\>)[r]$.

\begin{proposition}[\cite{kato scdv} 3.3 (1)]\label{scdv quotient}
Let $H$ be a normal subgroup of $G$, $\v$ an element in $R_C(G/H)$ and $\v'$ the image of $\v$ under the canonical map $R_C(G/H)\ra R_C(G)$. Then, we have
\begin{equation*}
s_G(\v')=s_{G/H}(\v)\quad \rm{and} \quad \sw_{\x}(\v')=\sw_{\x}(\v).
\end{equation*}
\end{proposition}

\begin{proposition}[\cite{kato scdv} 3.3 (2)]\label{ind}
Let $H$ be a subgroup of $G$. For any $\th\in R_C(H)$, we have
\begin{equation*}
s_G(\ind^G_H\th)=[G:H]\lt(s_H(\th)+\dim\th\cdot\di(L^H/K)\rt)
\end{equation*}
\begin{equation}\label{indsw}
\sw_{\x}(\ind^G_H\th)=[G:H]\lt(\sw_{\x}(\th)+(\dim\th-\<\th,1\>)\cdot\di(L^{H}/K)\rt).
\end{equation}

\end{proposition}
By \eqref{s1}, \eqref{sg1} and \eqref{relation diff}, equation \eqref{indsw} can be written as
\begin{equation}\label{indsw good}
\sw_{\x}(\ind^G_H\th)=[G:H]\lt(\sw_{\x}(\th)-(\dim\th-\<\th,1\>)\lt(\sum_{\s\in G-H} ([\dr \bar h]-[h-\s(h)])\rt)\rt).
\end{equation}

\begin{theorem}[\cite{kato scdv} 3.4]\label{class field}
For any $\c\in R_C(G)$, we have
\begin{equation*}
\sw_{\x}(\c)\in\skl\sub\slk\ot_{\Z}\widetilde{\Z}.
\end{equation*}
\end{theorem}
This is a generalization of Hasse-Arf's theorem. It can be reduced to the case where $G$ is cyclic of rank $p^s$ and $\c$ is 1-dimensional by the induction formula \ref{ind} and Brauer theorem. Then the proof relies on the higher dimensional class field theory of Kato (\cite{kato scdv} 3.6, 3.7).

\subsection{}\label{kcc}
For an element $\c\in R_C(G)$, the Swan conductor with differential values $\sw_{\x}(\c)$ is given by
\begin{equation*}
\sw_{\x}(\c)=-\sharp G(\dim_C\c-\<\c,1\>)[\dr \bar h]+\Delta,
\end{equation*}
where
\begin{equation*}
\Delta=\sum_{\s\in G-\{1\}}[h-\s(h)]\ot(\dim_C\c-\tr_\c(\s))+(\dim_C\c-\<\c ,1\>)\varepsilon(\x)\in R_L\ot_{\Z}\widetilde\Z.
\end{equation*}
From \eqref{vo} and \ref{class field}, we have $\sharp G[\dr \bar h]=[\dr \bar a]$ and $\Delta\in R_K$. Hence, we get
\begin{equation*}
\sw_{\x}(\c)=[\p^c]+[\D']-m[\dr \bar a]\in \skl,
\end{equation*}
where $\p$ is the uniformizer of $\OK$ fixed in (\ref{basic notes}), $c$ is an integer, $m=\dim_C\c-\<\c,1\>$ and $\D'\in F$ such that $[\p^c\D']=\D$. We define {\it Kato's characteristic cycle} of $\c$ and denote by $\kcc_{\x}(\c)$ the element
\begin{equation}\label{kcc formula}
\kcc_{\x}(\c)=\D'(\dr \bar a)^m\in (\O^1_F)^{\ot m}.
\end{equation}

\begin{remark}[\cite{kato scdv} 3.15]\label{ext swan 1}
If the extension $L/K$ is not of type (II), but there exists a subfield $K'$ of $L$ containing $K$ such that $K'/K$ is an unramified extension and $L/K'$ is of type (II), we define
\begin{equation*}
\sw_{\x}(\c)=\sw_{\x}(\res^G_{\gal(L/K')}\c).
\end{equation*}
Denote by $\cO_{K'}$ the integer ring of $K$, $\m_{K'}$ the maximal ideal of $\cO_{K'}$ and $F'$ the residue field of $\cO_{K'}$. Observe that $\sw_{\x}(\c)$ is fixed by $\gal(K'/K)$ and that
the $\gal(K'/K)$-invariant part of $F'\<\m_{K'}/\m^2_{K'},\ker(\O^1_{F'}\ra\O^1_E)\>$ is
$F\<\csk,\ker(\O^1_F\ra\O^1_E)\>$. Thus $\sw_{\x}(\c)$ is still contained in $\skl$.
\end{remark}
\begin{remark}[\cite{kato scdv} 3.16]\label{ext swan 2}
Let $A$ be an algebraically closed field of characteristic $\ell\notin\{0,p\}$. We denote by $A'$ an algebraic closure of the fraction field of the ring of Witt vectors $W(A)$. Let $\c$ be an element of $R_A(G)$ and let $\hat \c$ be a pre-image of $\c$ in $R_{A'}(G)$ (\cite{serre gr} 16.1 Th. 33). We denote by $\hat {\x}$ the $p$-th root of unity in $A'$ lifting of a primitive $p$-th root of unity $\x$ in $ A$. Then we put
\begin{equation*}
\sw_{\x}(\c)=\sw_{\hat{\x}}(\hat{\c}).
\end{equation*}
This definition is independent of the choice of $\hat{\c}$ because of (\cite{serre gr} 18.2 Th. 42) and \eqref{swan}.

\end{remark}

\section{Abbes-Saito's ramification theory}

\subsection{}\label{note}
Abbes and Saito defined two decreasing filtrations $\gk^r$ and $\gkl^r$ ($r\in\Q_{>0}$) of $G_K$ by closed normal subgroups called the ramification filtration and the logarithmic ramification filtration, respectively (\cite{as i}, 3.1, 3.2).
\subsection{}
We denote by $G_K^0$ the group $G_K$. For any $r\in \Q_{\geqslant 0}$, we put
\begin{equation*}
G_K^{r+}=\ol{\bigcup_{s\in\Q_{>r}}G_K^s} \quad\mathrm{and}\quad \gr ^r G_K=G_K^r/G^{r+}_K.
\end{equation*}
Let $L$ be a finite separable extension of $K$. For a rational number $r\geqslant 0$, we say that the ramification of $L/K$ is bounded by $r$ (resp. by $r+$) if $ G_K^r$ (resp. $G_K^{r+}$) acts trivially on $\Hom_K(L, \ol K)$ via its action on $\ol K$. We define the {\it conductor} $c$ of $L/K$ as the infimum of rational numbers $r>0$ such that the ramification of $L/K$ is bounded by $r$. Then $c$ is a rational number and $L/K$ is bounded by $c+$ (\cite{as i} 6.4). If $c>0$, the ramification of $L/K$ is not bounded by $c$.

\subsection{}
We denote by $\gkl^0$ the inertia subgroup of $\gk$. For any $r\in \Q_{\geqslant 0}$, we put
\begin{equation*}
\gkl^{r+}=\ol{\bigcup_{s\in\Q_{>r}}\gkl^s} \quad\mathrm{and}\quad \grl=\gkl^r\big/\gkl^{r+}.
\end{equation*}
By (\cite{as i} 3.15), $P=\gkl^{0+}$ is the wild inertia subgroup of $G_{K}$, i.e. the $p$-Sylow subgroup of $G_{K,\log}$.
Let $L$ be a finite separable extension of $K$. For a rational number $r\geqslant 0$, we say that the logarithmic ramification of $L/K$ is bounded by $r$ (resp. by $r+$) if $\gkl^r$ (resp. $\gkl^{r+}$) acts trivially on $\Hom_K(L, \ol K)$ via its action on $\ol K$. We define the {\it logarithmic conductor} $c$ of $L/K$ as the infimum of rational numbers $r>0$ such that the ramification of $L/K$ is bounded by $r$. Then $c$ is a rational number and $L/K$ is bounded by $c+$ (\cite{as i} 9.5). If $c>0$, the ramification of $L/K$ is not bounded by $c$.

\begin{theorem}[\cite{as ii} Th. 1]\label{center}
For every rational number $r>0$, the group $\grl$ is abelian and is contained in the center of $P/\gkl^r$.
\end{theorem}

\begin{lemma}[\cite{katz} 1.1]\label{slope decom lemma}
Let $M$ be a $\Z[\f{1}{p}]$-module on which $P=\gkl^{0+}$ acts through a finite discrete quotient, say by $\r: P\ra \aut_{\Z}(M)$. Then,
\begin{itemize}
\item[(i)]
The module $M$ has a unique direct sum decomposition
\begin{equation}\label{slopedecom}
M=\bp_{r\in\Q_{\geqslant0}} M^{(r)}
\end{equation}
into $P$-stable submodules $M^{(r)}$, such that $M^{(0)}=M^P$ and for every $r>0$,
\begin{equation*}
(M^{(r)})^{\gkl^r}=0\quad \mathrm{and}\quad (M^{(r)})^{\gkl^{r+}}=0.
\end{equation*}

\item[(ii)] If $r>0$, then $M^{(r)}=0$ for all but the finitely many values of $r$ for which $\r(\gkl^{r+})\neq\r(\gkl^r)$.

\item[(iii)] For any $r\geqslant 0$, the functor $M\ma M^{(r)}$ is exact.

\item[(iv)] For $M$, $N$ as above, we have $\Hom_{P-mod}(M^{(r)}, N^{(r')})=0$ if $r\neq r'$.

\end{itemize}
\end{lemma}
The decomposition \eqref{slopedecom} is called the {\em slope decomposition} of $M$. The values $r\geqslant0$ for which $M^{(r)}\neq 0$ are called the {\em slopes} of $M$. We say that $M$ is {\it isoclinic} if it has only one slope.

\subsection{}\label{l L psi}
In the following of this section, we fix a prime number $\ell$ different from $p$, a local $\Z_{\ell}$-algebra $\L$ which is of finite type as a $\Z_{\ell}$-module and a non-trivial character $\psi_0:\F_p\ra \L^{\ti}$.

\begin{lemma} [\cite{as rc} 6.7]\label{center char decomp}
Let $M$ be a $\L$-module on which $P$ acts $\L$-linearly through a finite discrete quotient, which is isoclinic of slope $r>0$. So the $P$ action on $M$ factors through the group $P/\gkl^{r+}$.
\begin{itemize}
\item[(i)]
Let $X(r)$ be the set of isomorphism classes of finite characters $\c:\grl\ra\L^{\ti}_{\c}$ such that $\L_{\c}$ is a finite \'etale $\L$-algebra, generated by the image of $\c$, and having a connected spectrum. Then $M$ has a unique direct sum decomposition
\begin{equation}\label{cen char decom form}
M=\bp_{\c\in X(r)}M_{\c}.
\end{equation}
Each $M_{\c}$ is a $P$ stable sub-$\L$-module such that $\L[\gkl^r]$ acts on  $M_{\c}$ through $\L_{\c}$.

\item[(ii)]
There are finitely many characters $\c\in X(r)$ for which $M_{\c}\neq0$.
\item[(iii)]
Fix $\c\in X(r)$, for all isoclinic $M$ of slope $r$, the functor $M\ra M_{\c}$ is exact.

\item[(iv)]
For $M$, $N$ as above, we have $\Hom_{\L}(M_{\c}, N_{\c'})=0$ if $\c\neq \c'$.

\end{itemize}
\end{lemma}
The decomposition \eqref{cen char decom form} is called the {\it central character decomposition} of $M$. The characters $\c: \grl\ra\L^{\ti}_{\c}$ for which $M_{\c}\neq0$ are called the central characters of $M$ (\cite{as rc} 6.8).

Let $P_0$ be a finite discrete quotient of $P/\gkl^{r+}$ through which $P$ acts on $M$ and let $C_0$ be the image of $\grl$ in $P_0$. By \ref{center}, we know that $C_0$ is contained in the center of $P_0$. The connected components of $\spec(\L[C_0])$ correspond to the isomorphism classes of characters $\c:C_0\ra \L^{\ti}_{\c}$, where $\L_{\c}$ is finite \'etale $\L$-algebra, generated by the image of $\c$, and having a connected spectrum. If $p^nC=0$, and $\L$ contains a primitive $p^n$-th root of 1, then $\L_{\c}=\L$ for every $\c$ such that $M_{\c}\neq 0$.

\begin{lemma}[\cite{katz} 1.4, \cite{as rc} 6.10]\label{slope center decom p to 0}
Let $A$ be a $\L$-algebra and $M$ a left $A$-module on which $P$ acts $A$-linearly through a finite discrete quotient. Then,
\begin{itemize}
\item[(i)]
In the slope decomposition $M=\bp_r M^{(r)}$, each $M^{(r)}$ is a sub-$A$-module of $M$. For any $A$-algebra $B$, the decomposition of $B\ot_A M$ is given by $B\ot_A M=\bp_r(B\ot_A M{(r)})$.
\item[(ii)]
If $M$ is isoclinic, then in the central character decomposition $M=\bp_{\c}M_{\c}$, each $M_{\c}$ is a sub-$A$-module of $M$. For any $A$-algebra $B$, the central character decomposition of $B\ot_A M$ is given by $B\ot_A M=\bp_{\c}(B\ot_A M_{\c})$.
\end{itemize}
\end{lemma}

\subsection{}\label{isogeny V}
Let $V$ be a finite dimensional $\ol F$-vector space and we denote by $V^*$ its dual space. We consider $V$ as a smooth abelian algebraic group over $\ol F$, i.e. $\spec (\sym(V^*))$. Let $\p_1^{\alg}(V)$ be the quotient of $\p_1^{\ab}(V)$ classifying \'etale isogenies. Then $\p_1^{\alg}(V)$ is a profinite group killed by $p$ and the group $\Hom(\p_1^{\alg}(V),\F_p)$ is canonical identified with the dual space $V^*$ by pulling-back the Lang's isogeny $\mathbb{A}^1\ra\mathbb{A}^1:\;t\ma t^p-t$ by linear forms (cf. \cite{wrcb} 1.19).

\subsection{}\label{iso}
For the rest of this section, we assume that $F$ is of finite type over a perfect subfield $F_0$. We define the $F$-vector space $\O^1_F(\log)$ by
\begin{equation*}
\O^1_F(\log)=(\O^1_{F/F_0}\op(F\ot_{\Z} K^{\ti}))/(\dr \bar a-\bar a\ot a;\;a\in\OK^{\ti}).
\end{equation*}
Then we have an exact sequence of finite dimensional $F$-vector spaces
\begin{equation} \label{Omegalog}
0\lra\O^1_F\lra\O^1_F(\log)\xra{\mathrm{res}}F\lra 0,
\end{equation}
where $\mathrm{res}((0,a\ot b))=a\cdot v(b)$ for $a\in F$ and $b\in K^{\ti}$. If $K$ has characteristic $p$, we put
\begin{equation*}
\what\O^1_{\OK/F_0}=\plim_n\O^1_{(\OK/\m_K^n)/F_0}.
\end{equation*}
We have an exact sequence of $F$-vector spaces
\begin{equation}\label{hatO exact p}
0\ra\m_K/\m_K^2\ra \what\O^1_{\OK/F_0}\ot_{\OK}F\ra\O^1_F\ra 0.
\end{equation}
If $K$ has characteristic zero and $p$ is not a uniformizer of $\OK$, we denote by $\cO_{K_0}$ the ring of Witt vectors $W(F_0)$ regarded as a sub-algebra of $\OK$. Then, we put
\begin{equation*}
\what\O^1_{\OK/\cO_{K_0}}=\plim_n\O^1_{(\OK/\m_K^n)/\cO_{K_0}}.
\end{equation*}
We have an exact sequence of $F$-vector spaces
\begin{equation}\label{hatO exact 0}
0\ra\m_K/\m_K^2\ra \what\O^1_{\OK/\cO_{K_0}}\ot_{\OK}F\ra\O^1_F\ra 0.
\end{equation}

For any rational number $r$, we put
\begin{equation*}
\m^r_{\ol K}=\{x\in \ol K\,|\,v(x)\geqslant r\},\ \ \ \m^{r+}_{\ol K}=\{x\in \ol K\,|\,v(x)> r\},
\end{equation*}
\begin{eqnarray}
\thlogr&=&\Hom_F\big(\O^1_F(\log), \m^r_{\ol K}/\m^{(r+)}_{\ol K}\big),\nonumber\\
\xir&=&\Hom_F\big(\O^1_F, \m^r_{\ol K}/\m^{(r+)}_{\ol K}\big)\label{xi}.
\end{eqnarray}
When $K$ has characteristic $p$ (resp. characteristic zero and $p$ is not a uniformizer of $\OK$), for any rational number $r>0$, we denote by $\thnonr$ the $\ol F$-vector space
\begin{equation}\label{thnon}
\thnonr=\Hom_F\big(\what\O^1_{\OK/F_0}\ot_{\OK}F, \m^r_{\ol K}/\m^{(r+)}_{\ol K}\big)
\end{equation}
\begin{equation*}
\big(\text{resp.}\ \ \ \thnonr=\Hom_F\big(\what\O^1_{\OK/\cO_{K_0}}\ot_{\OK}F, \m^r_{\ol K}/\m^{(r+)}_{\ol K}\big)\ \ \big).
\end{equation*}
By \eqref{Omegalog}, \eqref{hatO exact p} and \eqref{hatO exact 0}, when $p$ is not a uniformizer of $K$, we have homomorphisms
\begin{equation*}
\thlogr\ra\xir\ra\thnonr.
\end{equation*}

By (\cite{as ii} 5.12), we have a canonical surjection
\begin{equation}\label{pi1 gr}
\p_1^{\ab}(\thlogr)\ra\grl.
\end{equation}

\begin{theorem}[\cite{saito cc} 1.24,\cite{as iii} Th. 2]\label{isogeny}
For every rational number $r>0$, the canonical surjection \eqref{pi1 gr} factors through the quotient $\p_1^{\alg}(\thlogr)$. In particular, the abelian group $\grl$ is killed by $p$ and the surjection \eqref{pi1 gr} induces an injective homomorphism
\begin{equation}\label{rsw}
\rsw:\Hom(\grl, \F_p)\ra\Hom_{\ol F}(\m^r_{\ol K}/\m^{r+}_{\ol K}, \O^1_F(\log)\ot \ol F).
\end{equation}
\end{theorem}
The morphism \eqref{rsw} is called the {\em refined Swan conductor}.

\subsection{}\label{cc}
Let $M$ be a free $\L$-module of finite type on which $P$ acts $\L$-linearly through a finite discrete quotient. Let
\begin{equation*}
M=\bp_{r\in\Q_{\geqslant 0}}M^{(r)}
\end{equation*}
be the slope decomposition of $M$ and for each rational number $r>0$, let
\begin{equation*}
M^{(r)}=\bp_{\c\in X(r)}M^{(r)}_{\c}
\end{equation*}
be the central character decomposition of $M^{(r)}$. We notice that each $M^{(r)}_{\c}$ is a free $\L$-module. Enlarging $\L$, we may assume that for all rational number $r>0$ and $\c\in X(r)$, $\L=\L_{\c}$ (\ref{center char decomp}). Each $\c$ factors uniquely through (\ref{l L psi})
\begin{equation*}
\gr^r_{\log}G_K\ra\F_p\xra{\psi_0}\L^{\ti}.
\end{equation*}
We denote abusively by $\c$ the induced character $\gr^r_{\log}G_K\ra \F_p$.
We define the {\it Abbes-Saito characteristic cycle} $\cc_{\psi_0}(M)$ of $M$ by
\begin{equation}\label{cc formula}
\cc_{\psi_0}(M)=\bt_{r\in\Q_{> 0}}\bt_{\c\in X(r)}(\rsw(\c)\ot\p^{r})^{\dim_{\L} M^{(r)}_{\c}}\in (\O^1_F(\log)\ot_{F}\ol F)^{\ot \dim_{\L}M/M^{(0)}}.
\end{equation}

\section{Ramification of extensions of type (II)}
\subsection{}\label{ram type ii notation}
In this section, we assume that the residue field $F$ of $\OK$ is of finite type over a perfect field $F_0$ of characteristic $p$. Let $L$ be a finite Galois extension of $K$ of group $G$ and type (II) (\ref{type}), $\OL$ the integer ring of $L$ and $E$ the residue field of $\OL$. We denote by $p^n$ the degree of the residue extension $E/F$. We choose an element $h\in\OL$ such that its residue class $\bar h\in E$ is a generator of $E/F$. We have $\OL=\OK[h]$. Let $f(T)\in\OK[T]$ be the minimal polynomial of $h$:
\begin{equation}\label{f}
f(T)=T^{p^n}+a_{p^n-1}T^{p^n-1}+\cdots+a_0.
\end{equation}
Notice that $\bar a_0=\bar h^{p^n}\in F$. We put
\begin{equation}\label{conductor}
c=\sup_{\s\in G-\{1\}}v(h-\s(h))+\sum_{\s\in G-\{1\}} v(h-\s(h)),
\end{equation}
which is an integer $\geqslant p^n$.

For any rational number $r\geqslant0$, we denote by $G^r$ (resp. $G^r_{\log})$ the image of $G^r_K$ (resp. $G_{K,\log}^r$) in $G$ (\cite{as i} 3.1). Using the monogenic presentation $\OL=\OK[T]/(f(T))$, we obtain that, for any rational number $r>1$, $G^r=G^r_{\log}$(\cite{as i} 3.1, 3.2) and that the conductor of $L/K$ is $c$ (\cite{as i} 6.6). By \ref{isogeny}, the normal subgroup $G^c$ of $G$ is commutative and killed by $p$. In the following, we put $\sharp G^{c}=p^s$.

\subsection{}\label{ram type ii notation2}
For any integer $j\geqslant1$, we denote by $D^j$ the $j$-dimensional closed poly-disc of radius one over $K$ and by $\mathring D^j$ the $j$-dimensional open disc of radius one over $K$. For a rational number $r\geqslant 0$, the $j$-dimensional closed poly-disc of radius $r$ is denoted by $D^{j,(r)}=\{(x_1,...,x_j)\in D^j|v(x_i)\geqslant r\}$. Let
\begin{equation*}
\tilde f: D^1\ra D^1,\ \ \ x\ma f(x),
\end{equation*}
be the morphism induced by $f$.
For any rational number $r\geqslant0$, it is easy to see that $\tilde f^{-1}(D^{1,(r)})$ is a disjoint union of closed discs with the same radius, i.e. there exists a rational number $\r(r)\geqslant0$ such that
\begin{equation*}
\tilde f^{-1}(D^{1,(r)})=\coprod_{1\leqslant j\leqslant i}\lt(x_j+D^{1,(\r(r))}\rt),
\end{equation*}
where the $x_j$'s are zeros of $f$.
The function $\r:\Q_{\geqslant 0}\ra\Q_{\geqslant 0}$ is called the {\em Herbrand function} of the extension $L/K$. By (\cite{as ii} 6.6), we have $\r(c)=\sup_{\s\in G-\{1\}}v(h-\s(h))$ and
\begin{equation}\label{property Gc}
G^c=\{\s\in G;v(h-\s(h))\geqslant\r(c)\}.
\end{equation}

\subsection{}\label{u}
We denote by $u$ the map
\begin{equation}\label{def u}
u:G\ra E, \ \ \ \s\ma\lt\{\begin{array}{ll}
 u_{\s}=\ol{\lt(\f{h-\s(h)}{\p^{v(h-\s(h))}}\rt)},&\text{if}\ \ \s\neq 1,\\
u_{\s}=0,&\text{if}\ \ \s=1.
\end{array}\rt.
\end{equation}
The restriction $u|_{G^c}:G^c\ra E$ of $u$ to $G^c$ is an injective homomorphism of groups. Indeed, for any $\s\in G^c-\{1\}$, we have $v(h-\s(h))=\r(c)$. Hence, for $\s_1,\s_2\in G^c$, we have
\begin{eqnarray}
u_{\s_1\s_2}=\ol{\lt( \f{h-\s_1\s_2(h)}{\p^{\r(c)}}       \rt)}
=\ol{\lt( \f{h-\s_1(h)+\s_1(h-\s_2(h))}{\p^{\r(c)}}       \rt)}=u_{\s_1}+u_{\s_2}.\nonumber
\end{eqnarray}

\begin{proposition}\label{isogeny f}
The polynomial $f_{c}(T)=f(\p^{\r(c)}T+h)/\p^{c}\in L[T]$ has integral coefficients. Its reduction $\ol {f_{c}}\in E[T]$ is an additive polynomial of degree $p^s=\sharp G^{c}$ with a non-zero linear term.

\end{proposition}
\begin{proof}
We have
\begin{equation*}
f_{c}(T)=T\prod_{\s\in G-\{1\}}\f{\p^{\r(c)}T+h-\s(h)}{\p^{v(h-\s(h))}}\in \OL[T].
\end{equation*}
Hence
\begin{equation}\label{f_g}
\ol {f_{c}}(T)=T\prod_{\s\in G-\{1\}}\ol{\lt(\f{\p^{\r(c)}T+h-\s(h)}{\p^{v(h-\s(h))}}\rt)}=\prod_{\s\in G-G^{c}}u_{\s}\cdot\prod_{\s\in G^{c}}\lt(T+u_{\s}\rt).
\end{equation}
Choose an $\F_p$-basis $\t_1,...,\t_s$ of $G^{c}$, we get
\begin{eqnarray*}
\prod_{\s\in G^{c}}\lt(T+u_{\s}\rt)&=&\prod_{(j_1,...,j_s)\in \F^s_p}(T+j_1u_{\t_{1}}+\cdots+j_s u_{\t_{s}}).
\end{eqnarray*}
We conclude by the lemma below.
\end{proof}

\begin{lemma}\label{isogeny lemma}
Let $C$ be a field of characteristic $p$. For any integer $r\geqslant0$, let $x_1,...,x_r$ be $r$ elements of $C$ such that for any $(j_1,...,j_r)\in \F^n_p-\{0\}$, $j_1x_1+\cdots+j_rx_r\neq 0$. Then we have
\begin{equation}\label{math ind}
\prod_{(j_1,...,j_r)\in \F^n_p}(T+j_1x_1+\cdots+j_rx_r)=T^{p^r}+\l_{r-1}T^{p^{r-1}}+\cdots+\l_1T^p+\l_0T\in C[T],
\end{equation}
where
\begin{equation*}
\l_0=\prod_{(j_1,...,j_r)\in \F^r_p-\{0\}}(j_1x_1+\cdots+j_rx_r)\neq 0.
\end{equation*}
\end{lemma}
\begin{proof}
We proceed by induction on $r$. If $r=1$,
\begin{equation*}
\prod_{j_1\in \F_p}(T+j_1x_1)=T^p-x_1^{p-1}T,
\end{equation*}
which satisfies \eqref{math ind}. Assume that \eqref{math ind} holds for $(r-1)$-tuples where $r\geqslant 2$, let $(x_1,...,x_r)\in C^r$ be as in the lemma. We put
\begin{equation*}
g_{r-1}(T)=\prod_{(j_1,...,j_{r-1})\in \F^{r-1}_p}(T+j_1x_1+\cdots+j_{r-1}x_{r-1}).
\end{equation*}
Then, we have
\begin{eqnarray*}
\prod_{(j_1,...,j_r)\in \F^r_p}(T+j_1x_1+\cdots+j_rx_r)&=&\prod_{j_r\in\F_p}(g_{r-1}(T+j_r x_r))\\
 &=&\prod_{j_r\in\F_p}(g_{r-1}(T)+j_r g_{r-1}(x_r))\\
 &=&g^p_{r-1}(T)-g^{p-1}_{r-1}(x_r)g_{r-1}(T),
\end{eqnarray*}
which satisfies \eqref{math ind} since $g_{r-1}$ does.
\end{proof}

In the following of this section, we assume that $p$ {\em is not a uniformizer of} $K$.

\begin{lemma}\label{p^2}
Suppose $c>2$. Then, for any $1\leqslant i \leqslant p^n-1$, we have $v(a_i)\geqslant 2$ \eqref{f}.
\end{lemma}

\begin{proof}
From the equation $f(T)=\prod_{\s\in G}(T-\s(h))$, for any $1\leqslant i\leqslant p^n-1$, we obtain
\begin{eqnarray}\label{decomp a}
a_i&=&(-1)^{(p^n-i)}\sum_{\{\s_1,...,\s_{p^n-i}\}\sube G} \s_{1}(h)\s_{2}(h)\cdots\s_{p^n-i}(h)\\\nonumber
 &=&(-1)^{(p^n-i)}\sum_{\{\s_1,...,\s_{p^n-i}\}\sube G}(\s_{1}(h)-h+h)\cdots(\s_{p^n-i}(h)-h+h)\\\nonumber
 &=&(-1)^{(p^n-i)}\lt( {p^n\choose i} h^{p^n-i}+{p^n-1\choose i} h^{p^n-i-1}\sum_{\s\in G}(\s(h)-h)+\Delta            \rt),
\end{eqnarray}
where $v(\Delta)\geqslant 2$. Since the integer ${p^n\choose i}$ is divided by $p$, $v({p^n\choose i}h^{p^n})\geqslant 2$. Hence it is sufficient to show that
\begin{equation*}
v\Bigg(\sum_{\s\in G}(\s(h)-h)\Bigg)\geqslant 2.
\end{equation*}

Assume first that for any $\s\in G-\{1\}$, $v(h-\s(h))=\r(c)$, i.e. $G=G^{c}$. It suffices to treat the case where $\r(c)=1$. In this case, $\sharp G=c>2$ \eqref{conductor}. From \ref{ram type ii notation}, $G$ is an $\F_p$-vector space of dimension $n$ and we choose an $\F_p$-basis $\t_1,...,\t_n$ of $G$. By \ref{u}, we have
\begin{eqnarray*}
\sum_{\s\in G}u_{\s}&=&\sum_{(j_1,...,j_n)\in \F^n_p}(j_1u_{\t_1}+\cdots+j_nu_{\t_n})\\
 &=&\f{p^n(p-1)}{2}(u_{\t_1}+\cdots+u_{\t_n})=0,
\end{eqnarray*}
which means that $v(\sum_{\s\in G-\{1\}}(\s(h)-h))\geqslant \r(c)+1= 2$. 

Assume next that for $\s\in G-\{1\}$, the $v(h-\s(h))$'s are not equal. Let $c'$ be the smallest jump of the ramification filtration of $G$ and let $\sharp(G^{c'+})=p^{n'}$ for some integer $n'<n$. Let $\varsigma_1=1,\varsigma_2,...,\varsigma_{p^{n-n'}}$ be liftings of all the elements of $G/G^{c'+}$ in $G$. Observe that for any $\varsigma\in G-G^{c'+}$ and $\s\in G^{c'+}$, we have $u_{\varsigma\s}=u_{\varsigma}$. Hence
\begin{equation*}
\sum_{\varsigma\in G-G^{c'+}} u_{\varsigma}=\sum_{j=2}^{p^{n-n'}} \sum_{\s\in G^{c'+}}u_{\varsigma_j}=p^{n'}\sum_{j=2}^{p^{n-n'}}u_{\varsigma_j}=0.
\end{equation*}
Hence $v(\sum_{\s\in G-G^{c'+}}(\s(h)-h))\geqslant 2$. Meanwhile, $v(\sum_{\s\in G^{c'+}}(\s(h)-h))\geqslant 2$, hence we obtain the inequality $v(\sum_{\s\in G}(\s(h)-h))\geqslant 2$.
\end{proof}

\begin{proposition}\label{thlog factor xi}
The composition of the canonical homomorphisms (\ref{isogeny})
\begin{equation*}
\p^{\alg}_1(\thlogc)\ra  \gr^c_{\log}G_K  \ra G^c
\end{equation*}
 factors through $\p^{\alg}_1(\xic)$ \eqref{xi}. In particular, for any non-trivial character $\c:G^c\ra \F_p$, we have $\rsw(\c)\in \O^1_F\ot_{F}\m^{-c}_{\ol K}/\m^{-c+}_{\ol K}$.
\end{proposition}
The proof of this proposition is given in \ref{proof of proposition}.

\subsection{}\label{notation for rsw}
For a non-trivial character  $\c:G^{c}\ra \F_p$, we denote by $\bar f_{c,\c}(T)$ the polynomial (\ref{u})
\begin{equation}\label{notation ker c}
\bar f_{c,\c}(T)=\prod_{\s\in \ker\c}(T+u_{\s})\in \ol F[T],
\end{equation}
and by $\t\in G^{c}$ a lifting of $1\in \F_p$. Recall that $\bar f_{c,\c}$ is an additive polynomial with a non-zero linear term (\ref{isogeny lemma}), and that $\bar f_{c,\c}(u_{\t})$ is independent of the choice of $\t$.

\begin{theorem}\label{theorem rsw}
For any non-trivial character $\c:G^c\ra \F_p$, the refined Swan conductor $\rsw(\c)$ is given by
\begin{equation*}
\rsw(\c)=-\dr \bar a_0\ot\f{\p^{-c}}{\lt(\prod_{\s\in G-G^{c}}u_{\s}\rt)\bar f_{c,\c}^p(u_{\t})}\in \O^1_F\ot_{F}\m^{-c}_{\ol K}/\m^{-c+}_{\ol K}.
\end{equation*}
\end{theorem}
The proof of this theorem is given in \ref{proof of theorem}.

\begin{corollary}\label{cc in O}
Let $M$ be a finite dimensional $\L$-vector space with a non-trivial linear $G$-action. Then, with the notation of $\ref{l L psi}$, we have \eqref{cc formula}
\begin{equation*}
\cc_{\psi_0}(M)\in (\O^1_F\ot_F \ol F)^{\ot r},
\end{equation*}
where $r=\dim_{\L}M/M^{(0)}$ (\ref{slope decom lemma}).
\end{corollary}

\section{Tubular neighborhoods and normalized integral models}
\subsection{}
Let $R$ be an $\OK$-algebra. Following (\cite{as ii} 1), we say that $R$ is formally of finite type over $\OK$ if it is semi-local with radical $\m_R$, $\m_R$-adically complete, Noetherian and if the quotient $R/\m_R$ is of finite type over $F$. We say that $R$ is topologically of finite type over $\OK$ if it is $\p$-adically complete, Noetherian and if the quotient $R/\p R$ is of finite type over $F$.

\subsection{}\label{dejong}
We denote by $\afs$ the category of affine Noetherian adic formal schemes $\fX$ over $\spf(\OK)$ such that the closed sub-scheme $\fX_{\red}$ defined by the largest ideal of definition of $\fX$, is a scheme of finite type over $\spec(F)$. Let $A$ be a finite flat algebra over $\OK$, and  $i:\spf(A)\ra\fX$ a closed immersion in $\afs$. For any rational number $r>0$, following (\cite{dejong} 7.1 and \cite{am} 2.1), we associate to $i$ a $K$-affinoid variety $X^r$, called the {\em tubular neighborhood of $i$ of thickening $r$}, as follows. Let $\fX=\spf(\cA)$, $I$ be the ideal of $\cA$ which defines the immersion $i$ and $t,s>0$ be two integer such that $r=t/s$. Let $\cA\<I^s/\p^t\>$ be the $\p$-adic completion of the subalgebra of $\cA\ot_{\OK}K$ generated by $\cA$ and $f/\p^t$ for $f\in I^s$. Then $\cA\<I^s/\p^t\>\ot_{\OK} K$ is a $K$-affinoid algebra which depends only on $r$. We denote by $X^r$ the $K$-affinoid variety $\Sp(\cA\<I^s/\p^t\>\ot_{\OK} K)$.
For rational numbers $r'>r>0$, there exists a canonical morphism $X^{r'}\ra X^r$ which makes $X^{r'}$ a rational sub-domain of $X^r$. The admissible union of the affinoid spaces $X^{r}$ for $r\in\Q_{\geqslant 0}$ is a quasi-separated rigid variety over $K$.

\begin{proposition}[Finiteness theorem of Grauert-Remmert, \cite{bgr} 6.4.1/3, \cite{as i} 4.2]\label{fini theo}
Let $\cR$ be a geometrically reduced $K$-affinoid algebra. Then, there exists a finite separable extension $K'$ of $K$ such that the supremum norm unit ball (\cite{bgr} 3.8.1)
\begin{equation}\label{model over K'}
\cR_{\cO_{K'}}=\{f\in \cR\ot_K K';|f|_{\sup}\leqslant 1\}\sube\cR\ot_K K'
\end{equation}
has a reduced geometric closed fiber $\cR_{\cO_{K'}}\ot_{\cO_{K'}}\ol F$. Moreover, the formation of $\cR_{\cO_{K'}}$ commutes with any finite extension of $K'$.
\end{proposition}

\subsection{}\label{def nim}
Let $\cR$ be a geometrically reduced $K$-affinoid algebra. We consider the collection of $\cO_{K'}$-formal scheme $\spf(\cR_{\cO_{K'}})$, where $K'$ and $\cR_{\cO_{K'}}$ are as in \ref{fini theo}, as a unique model of $\Sp(\cR)$ over $\cO_{\ol K}$. We call it the {\it normalized integral model} over $\cO_{\ol K}$. We say that the the normalized integral model of $\Sp(\cR)$ is defined over $K'$ if the supremum norm unit ball $\cR_{\cO_{K'}}$ has a reduced geometric special fiber. We call this reduced geometric special fiber over $\ol F$ the special fiber of the normalized integral model of $\Sp(\cR)$ over $\cO_{\ol K}$.

\begin{proposition}[\cite{as i} 4.4]\label{iso sp gen}
Let $X$ be a geometrically reduced affinoid variety over $K$, $\fX$ its normalized integral model over $\cO_{\ol K}$ and $\ol\fX$ the special fiber of $\fX$. Then the set of geometric connected components of $X$ and $\ol\fX$ are isomorphic.
\end{proposition}

\subsection{}\label{geo monodromy}
Let $X$ be a geometrically reduced affinoid variety over $K$, $\fX$ its normalized integral model over $\cO_{\ol K}$ and $\ol\fX$ the special fiber of $\fX$. If $\fX$ is defined over a finite Galois extension $K'$ of $K$, we denote by $\fX_{\cO_{K'}}$ the normalized integral model of $X$ over $\cO_{K'}$. The natural $K'$-semi-linear action of $G_K$ on $X\ot_K K'$ extends to an $\cO_{K'}$-semi-linear action of $G_K$ on $\fX_{\cO_{K'}}$. If $K''$ is another finite Galois extension of $K$ containing $K'$, then $\fX'_{\cO_{K''}}=\fX_{\cO_{K'}}\ot_{\cO_{K'}}\cO_{K''}$ and the semi-linear action of $G_K$ on both sides are compatible. Hence, it induces an $\ol F$-semi-linear action of $G_K$ on the special fiber $\ol\fX$, called the {\it geometric monodromy} (\cite{as i} 4.5).

\section{Isogenies associated to extensions of type (II): the equal characteristic case}
\subsection{}\label{first pr char p}
In this section, we assume that $K$ has characteristic $p$ and that the residue field $F$ of $\OK$ is of finite type over a perfect field $F_0$. For an object $L$ of $\fek$ and an integer $r\geqslant 1$, we denote by $(\cO_{L}/\m^r_{L})\what\ot_{F_0} \OK$ the completion of $(\cO_{L}/\m^r_{L})\ot_{F_0} \OK$ relatively to the kernel of the homomorphism
\begin{equation}\label{surj first}
(\cO_{L}/\m^r_{L})\ot_{F_0} \OK \ra \cO_{L}/\m^r_{L},\ \ \ a\ot b\ra ab,
\end{equation}
and by $\cO_{L}\wdhat{\ot}_{F_0}\OK$ the projective limit
\begin{equation*}
\plim_r (\cO_{L}/\m^r_{L})\what\ot_{F_0} \OK.
\end{equation*}
We will always consider $\cO_{L}\wdhat{\ot}_{F_0}\OK$ as an $\OK$-algebra by the homomorphism
\begin{equation}\label{OK alg}
\OK\ra\cO_{L}\wdhat{\ot}_{F_0}\OK,\ \ \ u\ma 1\ot u,
\end{equation}
(in the following, we always abbreviate $1\ot u$ by $u$) and we will consider it as an $\cO_{L}$-algebra by
\begin{equation*}
\cO_{L}\ra \cO_{L}\wdhat{\ot}_{F_0}\OK,\ \ \ v\ma v\ot 1.
\end{equation*}
There is a canonical surjective homomorphism
\begin{equation}\label{surj third}
\cO_{L}\wdhat{\ot}_{F_0}\OK\ra\cO_{L}
\end{equation}
induced by the surjections \eqref{surj first}. We denote by $I_{L}$ its kernel.

\begin{proposition}[\cite{as ii} 2.3]\label{A ot O char p}
Let $L$ be an object of $\fek$.
\begin{itemize}
\item[(i)]
The $\OK$-algebra $\cO_{L}\wdhat{\ot}_{F_0}\OK$ is formally of finite type and formally smooth over $\OK$ and the morphism $(\cO_{L}\wdhat{\ot}_{F_0}\OK)/\m_{\cO_{L}\swdhat{\ot}_{F_0}\OK}\ra \cO_{L}/\m_{\cO_{L}}$ \eqref{surj third} is an isomorphism.
\item[(ii)]
Any morphism $L\ra L'$ of $\fek$ induces an isomorphism
\begin{equation}\label{origin reps}
\cO_{L'}\ot_{\cO_{L}}(\cO_{L}\wdhat{\ot}_{F_0}\OK)\isora\cO_{L'}\wdhat{\ot}_{F_0}\OK.
\end{equation}

\end{itemize}
\end{proposition}

\subsection{}\label{nim char p}
Let $L$ be an object of $\fek$. By \ref{A ot O char p}, $\spf(\cO_{L}\wdhat{\ot}_{F_0}\OK)$ is an object of $\afs$ (\ref{dejong}). For any rational number $r>0$ and integer numbers $s,t>0$ such that $r=t/s$, we denote by $\cR^r_{L}$ the $K$-affinoid algebra
\begin{equation}
\cR^r_{L}=(\cO_{L}\wdhat{\ot}_{F_0}\OK)\<I_{L}^s/\p^t\>\ot_{\OK}K,
\end{equation}
by $X^r_{L}=\Sp(\cR^r_{L})$ the tubular neighborhood of thickening $r$ of the closed immersion $\spf(\cO_{L})\ra \spf(\cO_{L}\wdhat{\ot}_{F_0}\OK)$ \eqref{surj third}, (\ref{dejong}), which is smooth over $K$ (\cite{as ii} 1.7). By \ref{fini theo}, there exists a finite separable extension $K'$ of $K$ such that the normalized integral model of $X^r_L$ over $\cO_{\ol K}$ is defined over $K'$ (\ref{def nim}). We denote by $\cR^r_{L,\cO_{K'}}$ the supremum norm unit ball of $\cR^r_{L}\ot_K K'$ (\ref{model over K'}), by $\fX^r_L$ the normalized integral model of $X^r_L$ over $\cO_{\ol K}$ and by $\ol \fX^r_L$ the special fiber of $\fX^r_L$ (\ref{def nim}).

\subsection{}\label{X int r}
Let $m$ be the dimension of the $F$-vector space $\O^1_F$, which is finite. By (\cite{as ii} 1.14.3), there is an isomorphism of $\OK$-algebras
\begin{equation}\label{iso o ot o}
\OK[[T_0,...,T_m]]\isora \cO_{K}\wdhat{\ot}_{F_0}\OK,
\end{equation}
such that the composition of it and \eqref{surj third} $\OK[[T_0,...,T_m]]\isora \cO_{K}\wdhat{\ot}_{F_0}\OK\ra \OK$ maps $T_i$ to $0$. Here the $\OK$-algebra structure of $\cO_{K}\wdhat{\ot}_{F_0}\OK$ is as in \eqref{OK alg}. If $r$ is an integer $\geqslant 1$, we have an isomorphism of $K$-affinoid algebras
\begin{equation}\label{iso affinoid K}
K\<T_0/\p^r,...,T_m/\p^r\>\isora \cR^r_K.
\end{equation}
The normalized integral model $\fX^r_K$ is defined over $\OK$, and we have an isomorphism
\begin{equation}\label{iso nim K}
\OK\<T_0/\p^r,...,T_m/\p^r\>\isora (\cO_{K}\wdhat{\ot}_{F_0}\OK)\<I_K/\p^r\>=\cR^r_{K,\OK}.
\end{equation}
Hence the closed fiber $\ol\fX^r_K$ is isomorphic to the affine scheme
\begin{equation*}
\spec \ol F[T_0/\p^r,...,T_m/\p^r].
\end{equation*}
In general, for any rational number $r>0$, the $K$-affinoid variety $X^r_K$ is isomorphic to $D^{m+1,(r)}$ and the rigid space $X_K=\cup_{r>0}X^r_K$ is isomorphic to $\mathring D^{m+1}$ (\ref{ram type ii notation2}).

By (\cite{as ii} 1.13, 2.4), for any rational number $r>0$, there exists a canonical isomorphism $\ol\fX^r_K\isora\thnonr$ (\ref{thnon}) which is compatible with the geometric monodromy on $\ol\fX^r_K$ and the natural $G_K$-action on $\thnonr$ (via its action on $\m^r_{\ol K}/\m^{r+}_{\ol K}$). If $r$ is an integer, it is constructed as follows. Firstly, we have a natural ring isomorphism
\begin{equation}\label{alg C to X}
\bp_{i=0}^{\infty}I_K^i/I_K^{i+1}\ot_{\OK}\m^{-ir}_{K}/\m^{-ir+1}_{K}\ra \cR^r_{K,\OK}/\m_K\cR^r_{K,\OK},\ \ \ \ol b\ot \ol c\ma \ol{bc},
\end{equation}
by \eqref{iso o ot o} and \eqref{iso nim K}. Extending scalars, we have
\begin{equation}\label{X to C}
\ol\fX^r_K\isora \spec\lt(\bp_{i=0}^{\infty}I_K^i/I_K^{i+1}\ot_{\OK}\m^{-ir}_{\ol K}/\m^{-ir+}_{\ol K}\rt).
\end{equation}
Then, from (\cite{as ii} 1.14.3, 2.4), we have an isomorphism of free $\OK$-modules
\begin{equation}\label{Omega iso I}
\hO_{\OK/F_0}\ra I_K/I_K^2 ,\ \ \ \dr t\ma\ol{1\ot t-t\ot 1},
\end{equation}
which induces the isomorphism $\ol\fX^r_K\ra\thnonr$.

\subsection{}\label{p to G^r}
Let $L$ be a finite Galois extension of $K$ of group $G$ and conductor $r>1$. By (\cite{as i} 7.2), the natural action of $G$ on $\cO_{L}\wdhat\ot_{F_0}\OK$ induces an $\cO_{\ol K}$-linear action of $G$ on $\fX^r_L$ making it an \'etale $G$-tosor over $\fX^r_K$. In particular, $X^r_{L}$ and $\ol\fX^r_{L}$ are \'etale $G$-tosors of $ X^r_K$ and $\ol\fX^r_K$, respectively. The geometric monodromy action of $G_K$ on $\ol\fX^r_{L}$ (\ref{geo monodromy}) commutes with the action of $G$. Let $\ol\fX^r_{L,0}$ be a connected component of $\ol\fX^r_{L}$. The stabilizers of $\ol\fX^r_{L,0}$ via these two actions are $G^r$ and $G^r_K$, respectively. Then, we get an isomorphism $G^r\isora \aut(\ol\fX^r_{L,0}/\ol\fX^r_K)$ and a surjection $G^r_K\ra \aut(\ol\fX^r_{L,0}/\ol\fX^r_K)$ which implies that $G^r$ is commutative (cf. \cite{as ii} 2.15.1). Composing with $\ol\fX^r_K\isora\thnonr$, the \'etale covering $\ol\fX^r_{L,0}\ra\thnonr$ induces a surjective homomorphism (\cite{as ii} 2.15.1)
\begin{equation*}
\p^{\ab}_1(\thnonr)\ra \gr^rG_K\ra G^r.
\end{equation*}

\subsection{}
In the rest of this section, let $L/K$ be a finite Galois extension of type (II) and we take again the notation and assumptions of \ref{ram type ii notation} and \ref{ram type ii notation2}.
By \eqref{origin reps} and the proof of (\cite{as ii} 1.6), for any rational number $r>0$, we have an isomorphism
\begin{equation}\label{cart1}
\cR_K^r\ot_{\OK\swdhat{\ot}_{F_0}\OK}(\OL\wdhat{\ot}_{F_0}\OK) \isora \cR_L^r.
\end{equation}
It induces, for any rational numbers $r>r'>0$, an isomorphism
\begin{equation*}
\cR^{r}_K\ot_{\cR^{r'}_K}\cR_{L}^{r'}\isora \cR^{r}_L,
\end{equation*}
which gives a Cartesian diagram of rigid spaces
\begin{equation}\label{tub morphism}
\xymatrix{\relax
X^{r}_L\ar[d]\ar[r]&X_L\ar[d]\\
X^{r}_K\ar[r]&X_K}
\end{equation}
where $X_K=\bigcup_{r>0}X^r_K$ and $X_L=\bigcup_{r>0}X^r_L$.

We put
\begin{equation*}
\bmf(T)=T^{p^n}+(a_{p^n-1}\ot 1)T^{p^n-1}+\cdots+(a_0\ot 1)\in (\OK\wdhat{\ot}_{F_0}\OK)[T].
\end{equation*}
From \eqref{origin reps} and \eqref{cart1}, we have a surjection
\begin{equation*}
\t_L:\cR^r_K\<T\>\ra \cR^r_L,\ \ \ T\ma h\ot 1,
\end{equation*}
which induces an isomorphism that we denote abusively also by
\begin{equation}\label{key representation}
\t_L:\cR^r_K\<T\>/\bmf(T)\isora \cR^r_L.
\end{equation}
In other terms, we have a co-Cartesian diagram of homomorphisms of $\cR^r_K$-algebras
\begin{equation}\label{cocart}
\xymatrix{\relax
\cR^r_L&\cR^r_K\<T\>\ar[l]_{\t_L}\\
\cR^r_K\ar[u]&\cR^r_K\<T\>\ar[u]_{\phi}\ar[l]_{\t_K}
}
\end{equation}
where $\phi(T)=\bmf(T)$ and $\t_K(T)=0$. Hence, taking the union of the $K$-affinoid varieties associated to each of the $K$-affinoid algebras in \eqref{cocart} for $r\in\Q_{> 0}$, we have a Cartesian diagram
\begin{equation}\label{tub explicit}
\xymatrix{\relax
X_L\ar[d]\ar[r]^-(0.5){i_L}&X_K\ti D^1\ar[d]^{\bff}\\
X_K\ar[r]^-(0.5){i_K}&X_K\ti D^1}
\end{equation}
where $i_L$, $\bff$ and $i_K$ are the morphisms induced by $\t_L$, $\phi$ and $\t_K$.

\subsection{}\label{ab}
In the following, for any $0\leqslant i\leqslant p^n-1$, we denote by $\a_i$ the element $ a_i-a_i\ot 1\in I_K$ (\ref{first pr char p}). When the conductor $c>2$, for each $1\leqslant i\leqslant p^n-1$, $v(a_i)\geqslant 2$ (\ref{p^2}). Let $a'_i=\p^{-2}a_i\in\OK$. We denote by $\a'_i$ the element $a'_i-a'_i\ot 1\in I_K$ and by $\b$ the element $\p-\p\ot1\in I_K$. Then, we have
\begin{equation*}
\a_i=(a'_i-\a'_i)(2\p\b-\b^2)+\p^2\a'_i.
\end{equation*}
Since $\a'_i$, $\b\in I_K \sub\p^c\cR^c_{K,\OK}$, we have $\a_i\in \p^{c+1}\cR^{c}_{K,\OK}$.

When $c=2$, we have $p=2$, $\sharp G=2$, $\r(c)=1$ and $a''_1=\p^{-1} a_1\in \OK$. We denote by $\a''_1$ the element $ a''_1-a''_1\ot 1\in I_K$. Then we have
\begin{equation*}
\a_1=(a''_1-\a''_1)\b+\p \a''_1.
\end{equation*}
Since $\a''_1,\b\in\p^c\cR^c_{K,\OK}$, we have $\a_1\in \p^c\cR^{c}_{K,\OK}$, and $\ol{\a_1/\p^c}=\ol{a''_1\b/\p^c}\in \cR^{c}_{K,\OK}\big/\p\cR^{c}_{K,\OK}$.

We put
\begin{equation*}
\bmf_0(T)=\sum_{0\leqslant i\leqslant p^n-1}(\a_i/\p^c)\cdot T^i\in \cR^c_{K,\OK}[T].
\end{equation*}
We have
\begin{equation*}
\bmf(T)=f(T)-\sum_{0\leqslant i\leqslant p^n-1}\a_iT^i=f(T)-\p^{c} \bmf_0(T).
\end{equation*}

In the rest of this section, we fix an embedding $L\ra\ol K$.

\begin{proposition}\label{component}
The $K$-affinoid $X^{c}_L$ has $\sharp(G/G^{c})=p^{n-s}$ geometric connected components. Let $\s_1,...,\s_{p^{n-s}}$ be liftings of all the elements of $G/G^{c}$ in $G$. We have
\begin{equation}\label{disjoint union}
i_L(X^{c}_L)\sube\coprod_{1\leqslant j\leqslant p^{n-s}} X^{c}_K\ti (\s_j(h)+D^{1,(\r(c))})\sube X_K\ti D^1,
\end{equation}
and each disc of the disjoint union contains exact one geometric connected component of $X^{c}_L$.
\end{proposition}
\begin{proof}
By the Cartesian diagrams \eqref{tub morphism} and \eqref{tub explicit}, we have
\begin{equation*}
i_L(X^{c}_L)=\bff^{-1}(i_K(X^{c}))\sube X^{c}_K\ti D^1\sube X_K\ti D^1.
\end{equation*}
Taking in account the isomorphisms \eqref{iso affinoid K} and \eqref{iso nim K}, for any point
\begin{equation*}
(t_0,...,t_m, t)\in X^{c}_K\ti D^1-\coprod_{1\leqslant k\leqslant p^{n-s}} X^{c}_K\ti (\s_k(h)+D^{1,(\r(c))}),
\end{equation*}
we have $v(f(t))<c$ and $v((\a_i/\p^{c})(t_0,...t_m)t^i)\geqslant 0$. Hence
$v(f(t)-\p^{c} \bmf_0(t_0,...,t_m,t))<c$ which means $\bff(t_0,..,t_m,t)=(t_1,...,t_m,\bmf(t_0,...,t_m,t))\not\in i_K(X^c_K)$.
Thus \eqref{disjoint union} holds. By the proof of (\cite{as ii} 2.15), $X^{c}_L$ has exactly $p^{n-s}$ geometric connected components.
Moreover, for any $1\leqslant j\leqslant p^{n-s} $, $f(\s_j(h))-\p^{c} \bmf_0(0,...,0,\s_j(h))=0$, hence each disc $X^c_K\ti (\s_j(h)+D^{1,(\r(c))})$ contains at least one geometric connected component of $X^{c}_L$.
\end{proof}
In the following, we denote by $\ol\fX^c_{L,0}$ the connected component of $\ol\fX^c_L$ corresponding to the connected component $X^c_{L,0}$ of $X^c_L$ containing $(0,...,0,h)\in X^c_K\ti D^1$ defined over $L$.

\begin{proposition}\label{thnong isogeny}
There exists a canonical Cartesian diagram
\begin{equation}\label{cartesian diagram isogeny}
\xymatrix{\relax
\ol\fX^c_{L,0}\ar[d]\ar[r]^{\nu}&\A^1_{\ol F}\ar[d]^{\ol {f_c}}\\
\thnonc\ar[r]^{\mu}&\A^1_{\ol F}}
\end{equation}
where $\ol{f_{c}}$ is defined in \eqref{f_g}, such that if $\xi$ is the canonical coordinate of $\A^1_{\ol F}$, we have
\begin{equation*}
\mu^*(\xi)=\lt\{\begin{array}{ll}\dr a_0\ot\p^{-c},&\text{if}\ \ c>2,\\
(a''_1h\dr\p+\dr a_0)\ot \p^{-2},&\text{if}\ \ c=2.
\end{array}\rt.
\end{equation*}
Moreover, for any $\s\in G^c$, the following diagram
\begin{equation}\label{G action diagram}
\xymatrix{\relax
\ol \fX^c_{L,0}\ar[d]_{\s}\ar[r]^{\nu}&\A^1_{\ol F}\ar[d]^{d_{\s}}\\
\ol \fX^c_{L,0}\ar[r]^{\nu}&\A^1_{\ol F}}
\end{equation}
where $d_{\s}^{*}(\x)=\x-u_{\s}$ (\ref{u}), is commutative.
\end{proposition}
\begin{proof}

We consider the $K$-affinoid algebra $\cR^c_K$ (resp. $\cR^c_L$) as a sub-ring of the $L$-affinoid algebra $\cR^c_K \ot_K L$ (resp. $\cR^c_L \ot_K L$).
By \eqref{component}, we have
\begin{equation*}
X^c_{L,0}=i^{-1}_L(X^c_K\ti(h+D^{1,(\r(c))}))\cap X^c_L.
\end{equation*}
Hence $X^c_{L,0}$ is presented by the $L$-affinoid algebra
\begin{equation}\label{connect comp over L}
(\cR^{c}_L\ot_{K}L)\<T'\>/(\p^{\r(c)}T'+h-h\ot 1).
\end{equation}
By the isomorphism \eqref{key representation}, \eqref{connect comp over L} is isomorphic to
\begin{equation*}
(\cR^{c}_K\ot_{K}L)\<T,T'\>/(\bmf(T), \p^{\r(c)}T'+h-T),
\end{equation*}
which, after eliminating $T$ by the relation $\p^{\r(c)}T'+h-T=0$, is
\begin{equation}\label{connected comp}
(\cR^{c}_K\ot_{K}L)\<T'\>/(\bmf(\p^{\r(c)}T'+h)).
\end{equation}
In both cases, by \ref{isogeny f} and \ref{ab}, we have
\begin{equation*}
\bmf(\p^{\r(c)}T'+h)/\p^{c}\in \cR^{c}_{K,\OL}\<T'\>,
\end{equation*}
\begin{equation*}
\bmf(\p^{\r(c)}T'+h)/\p^{c+1}\notin\cR^{c}_{K,\OL}\<T'\>.
\end{equation*}
Then the image of $\cR^{c}_{K,\OL}\<T'\>$ by the canonical surjection
\begin{equation*}
(\cR^c_K\ot_KL)\<T'\>\ra(\cR^{c}_K\ot_{K}L)\<T'\>/(\bmf(\p^{\r(c)}T'+h)),
\end{equation*}
is
\begin{equation}\label{oB}
\cR^{c}_{K,\OL}\<T'\>/(\bmf(\p^{\r(c)}T'+h)/\p^{c}).
\end{equation}
Extending the scalars from $\OL$ to $\ol F$, we obtain the following $\ol F$-algebra:
\begin{itemize}
\item[(i)]
if $c>2$,
\begin{equation}\label{clos fiber >2}
   (\cR^{c}_{K,\OL}\ot_{\OL}\ol F)[T']/(\ol { f_{c}}(T')-\ol{\a_0/\p^{c}})~;
\end{equation}
\item[(ii)]
if $c=2$,
\begin{equation}\label{clos fiber =2}
(\cR^{c}_{K,\OL}\ot_{\OL}\ol F)[T']/(\ol{f_2}(T')- \ol{(\a_0+a_1''h\b)/\p^{2}}).
\end{equation}
\end{itemize}
From isomorphisms \eqref{alg C to X}, \eqref{Omega iso I} and the canonical exact sequence \eqref{hatO exact p}, we know that when $c>2$ (resp. $c=2$), $\ol{\a_0/\p^c}$ (resp. $\ol{(\a_0+a_1''h\b)/\p^{2}}$) is a non-zero linear term in $\ol F\ot_{\OL}\cR^{c}_{K,\OL}$.
Hence \eqref{clos fiber >2} and \eqref{clos fiber =2} are all reduced.
Then, by (\cite{as i} 4.1),
\begin{equation*}
\spf(\cR^{c}_{K,\OL}\<T'\>/(\bmf(\p^{\r(c)}T'+h)/\p^{c}))
\end{equation*}
is the normalized integral model of $X^c_{K,0}$ defined over $\OL$. Hence $\ol\fX^{c}_{L,0}$ is defined by the $\ol F$-algebra \eqref{clos fiber >2} (resp. \eqref{clos fiber =2}) when $c>2$ (resp. $c=2$). We put
\begin{equation*}
\nu:\ol\fX^c_{L,0}\ra\A^1_{\ol F}=\spec(\ol F[\x]),\ \ \ \nu^*(\x)=T'.
\end{equation*}
It follows form the isomorphism $\ol\fX^{c}_K\isora \thnonc$ (\ref{X int r}) that \eqref{cartesian diagram isogeny} is Cartesian.

For any $\s\in G^c$, let $y_{\s}(x)$ be a polynomial $b_rx^r +\cdots +b_0\in \OK[x]$, where $ r\leqslant p^n-1$, such that $y_{\s}(h)=(h-\s(h))/\p^{\r(c)}\in \OL$. We denote by $\mathbbm y_{\s}$ the polynomial
\begin{equation*}
\mathbbm y_{\s}(x)=(b_r\ot 1)x^r+\cdots+(b_0\ot 1)\in \cR^c_K[x].
\end{equation*}
The action of $\s$ on $\cR^c_K\<T\>/\bmf(T)$ (isomorphic to $ \cR^c_L$ \eqref{key representation}) is given by $:T\ma T-(\p^{\r(c)}\ot1)\mathbbm y(T)$.
Hence the action of $\s$ on \eqref{connected comp} is given by
\begin{equation*}
T'\ma T'-\mathbbm y_{\s}(\p^{\r(c)}T'+h)-((\p^{\r(c)}\ot1-\p^{\r(c)})/\p^{\r(c)})\mathbbm y_{\s}(\p^{\r(c)}T'+h)
\end{equation*}
and the induced action on \eqref{oB} is given by the same formula. Since $\p^{\r(c)}\ot1-\p^{\r(c)}\in \p^{c}\cR^c_{K,\OK}$ and $c>\r(c)$, the reduction of $(\p^{\r(c)}\ot1-\p^{\r(c)})/\p^{\r(c)}$ to the geometric special fiber is $0$.
For any $0\leqslant j\leqslant r$, $b_j\ot 1- b_j\in \p^c\cR^c_{K,\OK}$. Then, the reduction of $\mathbbm y_{\s}(\p^{\r(c)}T'+h)$ to the geometric special fiber is (\ref{u})
\begin{equation*}
\ol{\mathbbm y_{\s}(\p^{\r(c)}T'+h)}=\ol {y_{\s}(\p^{\r(c)}T'+h)}=\ol{y_{\s}(h)}=u_{\s}.
\end{equation*}
Hence, diagram \eqref{G action diagram char 0} is commutative.
\end{proof}

\section{Isogenies associated to extensions of type (II): the unequal characteristic case}

\subsection{}
In this section, we assume that $K$ has characteristic $0$ and that the residue field $F$ of $\OK$ is of finite type over a perfect field $F_0$. Let $K_0$ be the fraction field of the ring of Witt vectors $W(F_0)=O_{K_0}$ considered as a subfield of $K$. We denote by $m$ the dimension of the $F$-vector space $\O^1_F$, which is finite.

\subsection{}\label{present cartier}
Let $L$ be an object of $\fek$. We call an $\cO_{K_0}$-{\it presentation of Cartier type} of $\OL$ a pair $(\cA_L, j:\cA_L\ra \OL)$, where $\cA_L$ is a complete semi-local Noetherian $\cO_{K_0}$-algebra formally smooth of relative dimension $m+1$ over $\cO_{K_0}$ and $j$ a surjective homomorphism of $\cO_{K_0}$-algebra inducing an isomorphism $\cA_L/\m_{\cA_L}\isora \cO_L/\m_L$ such that the kernel of $j$ is generated by a non-zero divisor of $\cA_L$.

Let $L_1$, $L_2$ be two objects of $\fek$ and $(\cA_{L_1},j_1:\cA_{L_1}\ra\cO_{L_1})$, $(\cA_{L_2},j_2:\cA_{L_2}\ra\cO_{L_2})$ two $\cO_{K_0}$-presentations of Cartier type. A morphism $(g,\bmg)$ from $(\cA_{L_1},j_1)$ to $(\cA_{L_2},j_2)$ is a pair of $\cO_{K_0}$-homomorphisms $g:\cO_{L_1}\ra\cO_{L_2}$ and $ \bmg:\cA_{L_1}\ra\cA_{L_2} $ such that the diagram
\begin{equation}\label{A0 ra OL}
\xymatrix{\relax
\cA_{L_1}\ar[d]_{\bmg}\ar[r]^{j_1}&\cO_{L_1}\ar[d]^g\\
\cA_{L_2}\ar[r]^{j_2}&\cO_{L_2}}
\end{equation}
is commutative. We say that $(g,\bmg)$ is {\em finite and flat} if $\bmg$ is finite and flat and if the diagram \eqref{A0 ra OL} is co-Cartesian.

\begin{proposition}[\cite{as ii} 2.7, 2.8]\label{step 0 lemma}
\begin{itemize}
\item[(i)]
Any object of $\fek$ admits an $\cO_{K_0}$-presentation of Cartier type.
\item[(ii)]
Let $g:L_1\ra L_2$ be a morphism of $\fek$, and $(\cA_{L_1},j_1)$, $(\cA_{L_2},j_2)$ two $\cO_{K_0}$-presentations of Cartier type. Then there exist a morphism $\bmg:\cA_{L_1}\ra\cA_{L_2}$ such that $(g,\bmg)$ is a morphism of $\cO_{K_0}$-presentations of Cartier type.
\item[(iii)]
Let $g:L_1\ra L_2$ be a morphism of $\fek$ and $(g,\bmg)$ a morphism between $\cO_{K_0}$-presentations of Cartier type $(\cA_{L_1},j_1)$ and $(\cA_{L_2},j_2)$. If a uniformizer $\p_0$ of $K_0$ is not a uniformizer of any factor of $\cO_{L_1}$, then $(g,\bmg)$ is finite and flat.
\end{itemize}
\end{proposition}

\subsection{}\label{A ot O}
Let $L$ be an object of $\fek$, and $(\cA_{L},j:\cA_L\ra\OL)$ an $\cO_{K_0}$-presentation of Cartier type. We denote by $(\cA_L/\m^r_{\cA_L})\what\ot_{\cO_{K_0}}\OK$ the formal completion of $(\cA_L/\m^r_{\cA_L})\ot_{\cO_{K_0}}\OK$ relatively to the kernel of the homomorphism
\begin{equation}\label{surj first char 0}
(\cA_L/\m^r_{\cA_L})\ot_{\cO_{K_0}}\OK\ra \OL/\m^r_{\OL},\ \ \ a\ot b\ma ab,
\end{equation}
and by $\cA_L\wdhat\ot_{\cO_{K_0}}\OK$ the projective limit
\begin{equation}\label{AwdhatO}
\cA_L\wdhat\ot_{\cO_{K_0}}\OK=\plim_r((\cA_L/\m^r_{\cA_L})\what\ot_{\cO_{K_0}}\OK).
\end{equation}
We will always consider $\cA_L\wdhat\ot_{\cO_{K_0}}\OK$ as an $\OK$-algebra by the homomorphism
\begin{equation*}
\OK\ra\cA_L\wdhat\ot_{\cO_{K_0}}\OK,\ \ \ u\ma 1\ot u,
\end{equation*}
(in the following, we always abbreviate $1\ot u$ by $u$) and we will consider it as an $\cA_L$-algebra by
\begin{equation*}
\cA_{L}\ra\cA_L\wdhat\ot_{\cO_{K_0}}\OK,\ \ \ v\ma v\ot 1.
\end{equation*}
There is a canonical surjective homomorphism
\begin{equation}\label{surj third char 0}
\cA_L\wdhat\ot_{\cO_{K_0}}\OK\ra\cO_{L},
\end{equation}
induced by the surjections \eqref{surj first char 0}. We denote by $I_{L}$ its kernel.

\begin{proposition}[\cite{as ii} 2.9]
Let $L$ be an object of $\fek$, and $(\cA_{L},j:\cA_{L}\ra \cO_{L})$ an $\cO_{K_0}$-presentation of Cartier type. Then,
\begin{itemize}
\item[(i)]
The $\OK$-algebra $\cA_L\wdhat\ot_{\cO_{K_0}}\OK$ is formally of finite type and formally smooth over $\OK$ and the morphism
$\cA_L\wdhat\ot_{\cO_{K_0}}\OK/\m_{\cA_L\swdhat\ot_{\cO_{K_0}}\OK}\ra \OL/\m_{\OL}$ \eqref{surj third char 0} is an isomorphism.
\item[(ii)]
Let $L'$ be another object in $\fek$ and $(\cA_{L'},j':\cA_{L'}\ra \cO_{L'})$ an $\cO_{K_0}$-presentation of Cartier type. If a uniformizer $\p_0$ is not a uniformizer of any factor of $\OL$, then, any morphism $(\cA_{L},j)\ra(\cA_{L'},j')$ induces an isomorphism
\begin{equation}\label{key iso char 0}
\cA_{L'}\ot_{\cA_L}(\cA_L\wdhat\ot_{\cO_{K_0}}\OK)\isora \cA_{L'}\wdhat\ot_{\cO_{K_0}}\OK.
\end{equation}
\end{itemize}
\end{proposition}
\begin{proof}
Part (i) is proved in (\cite{as ii} 2.9). For part (ii), we may assume $L$ and $L'$ are fields. We denote by $e$ the ramification index of the extension $L'/L$. For any integer $r\geqslant 1$, we have the following canonical commutative diagram
\begin{equation*}
\xymatrix{
\cA_L\ar[r]^-(0.5){\mathrm{pr}_1}\ar[d]&(\cA_{L}/\m_{\cA_{L}}^r)\ot_{\cO_{K_0}}\OK\ar[d]\ar[r]^-(0.5){\eqref{surj first char 0}}&\OL/\m_L^r\ar[d]\\
\cA_{L'}\ar[r]^-(0.5){\mathrm{pr}'_1}&(\cA_{L'}/\m_{\cA_{L}}^r\cA_{L'})\ot_{\cO_{K_0}}\OK\ar[r]^-(0.5){g_{L'}}&\cO_{L}/\m_{L'}^{er}}
\end{equation*}
such that each square is Cartesian. We denote by $(\cA_{L'}/\m_{\cA_{L}}^r\cA_{L'})\what\ot_{\cO_{K_0}}\OK$ the formal completion of $(\cA_{L'}/\m_{\cA_{L}}^r\cA_{L'})\ot_{\cO_{K_0}}\OK$ relatively to the kernel of $g_{L'}$. Since $\cA_{L}$ is a Noetherian local ring, by \ref{step 0 lemma}(iii) and Nakayama's lemma, $\cA_{L'}$ is a finite free $\cA_L$-module. Then, we have
\begin{equation*}
\cA_{L'}\ot_{\cA_{L}}(\cA_{L}/\m_{\cA_{L}}^r)\what\ot_{\cO_{K_0}}\OK\isora(\cA_{L'}/\m_{\cA_{L}}^r\cA_{L'})\what\ot_{\cO_{K_0}}\OK.
\end{equation*}
After taking projective limit on both sides, we obtain
\begin{equation}\label{iso 1 char 0}
\cA_{L'}\ot_{\cA_{L}}(\cA_{L}\wdhat\ot_{\cO_{K_0}}\OK)\isora\plim_r((\cA_{L'}/\m_{\cA_{L}}^r\cA_{L'})\what\ot_{\cO_{K_0}}\OK).
\end{equation}
By the proof of (\cite{as ii} 2.7.3), we obtain that $\m_{\cA_{L'}}^e\sube\m_{\cA_L}\cA_{L'}\sube\m_{\cA_{L'}}$. Hence, for any integer $r\geqslant 1$, we have two surjections
\begin{equation*}
(\cA_{L'}/\m^{er}_{\cA_{L'}})\what\ot_{\cO_{K_0}}\OK\sra(\cA_{L'}/\m^{r}_{\cA_L}\cA_{L'})\what\ot_{\cO_{K_0}}\OK  \sra(\cA_{L'}/\m^{r}_{\cA_{L'}})\what\ot_{\cO_{K_0}}\OK,
\end{equation*}
which imply
\begin{equation}\label{iso 2 char 0}
\plim_r((\cA_{L'}/\m_{\cA_{L}}^r\cA_{L'})\what\ot_{\cO_{K_0}}\OK)\isora\cA_{L'}\wdhat\ot_{\cO_{K_0}}\OK.
\end{equation}
Combining \eqref{iso 1 char 0} and \eqref{iso 2 char 0}, we get (ii).
\end{proof}

\subsection{}\label{nim char 0}
Let $L$ be an object of $\fek$, and $(\cA_{L},j:\cA_{L}\ra \cO_{L})$ an $\cO_{K_0}$-presentation of Cartier type. We will introduce objects analogue of those defined in $\S7$, and denote them by the same notation. For any rational number $r>0$ and integer numbers $s,t>0$ such that $r=t/s$, we denote by $\cR^r_{L}$ the $K$-affinoid algebra
\begin{equation*}
\cR^r_{L}=(\cA_L\wdhat\ot_{\cO_{K_0}}\OK)\<I_{K}^s/\p^t\>\ot_{\OK}K,
\end{equation*}
by $X^r_{L}=\Sp(\cR^r_{L})$ the tubular neighborhood of thickening $r$ of the immersion
\begin{equation*}
\spf(\cO_{L})\ra \spf(\cA_L\wdhat\ot_{\cO_{K_0}}\OK),
\end{equation*}
which is smooth over $K$ (\cite{as ii} 1.7). By \ref{fini theo}, there exists a finite separable extension $K'$ of $K$ such that the normalized integral model of $X^c_L$ is defined over $K'$ (\ref{def nim}). We denote by $\cR^r_{L,\cO_{K'}}$ the supremum norm unit ball of $\cR^r_{L}\ot_K K'$ (\ref{model over K'}), by $\fX^r_L$ the normalized integral model of $X^r_L$ over $\cO_{\ol K}$ and by $\ol \fX^r_L$ the special fiber of $\fX^r_L$.

\subsection{}
In the following of this section, we assume that $p$ is not a uniformizer of $K$.
By (\cite{as ii} 1.14.3), there is an isomorphism of $\OK$-algebras
\begin{equation}\label{iso o ot o char 0}
\OK[[T_0,...,T_m]]\isora \cA_K\wdhat\ot_{\cO_{K_0}}\OK,
\end{equation}
such that the composition of it and \eqref{surj third char 0} $\OK[[T_0,...,T_m]]\ra \cA_K\wdhat\ot_{\cO_{K_0}}\OK\ra\OK$ maps $T_i$ to $0$. If $r$ is an integer $\geqslant 1$, we have an isomorphism of $K$-affinoid algebras
\begin{equation}\label{iso affinoid K char 0}
K\<T_0/\p^r,...,T_m/\p^r\>\isora \cR^r_K.
\end{equation}
The normalized integral model $\fX^r_K$ is defined over $\OK$, and we have an isomorphism
\begin{equation}\label{iso nim K char 0}
\OK\<T_0/\p^r,...,T_m/\p^r\>\isora (\cA_K\wdhat\ot_{\cO_{K_0}}\OK)\<I_K/\p^r\>=\cR^r_{K,\OK}.
\end{equation}
Hence the geometric closed fiber $\ol\fX^r_K$ is isomorphic to the affine scheme
\begin{equation*}
\spec \ol F[T_0/\p^r,...,T_m/\p^r].
\end{equation*}
In general, for any rational number $r>0$, the $K$-affinoid variety $X^r_K$ is isomorphic to $D^{m+1,(r)}$ and the rigid space $X_K=\bigcup_{r>0}X^r_K$ is isomorphic to $\mathring D^{m+1}$ (\ref{ram type ii notation2}).

By (\cite{as ii} 2.11.2), we have an isomorphism
\begin{equation}\label{I/I2 iso O char 0}
(I_K/I^2_K)\ot_{\OK}F\ra \what\O^1_{\OK/\cO_{K_0}}\ot_{\OK} F,
\end{equation}
such that for any $x\in\OK$ and $\wt x$ a lifting in $\cA_K$, the image of
$(\ol{1\ot x-\widetilde x\ot 1})\ot 1$ is $\dr x \ot 1.$
From (\cite{as ii} 1.14.3, 2.11.2), for any rational number $r>0$, the inverse of \eqref{I/I2 iso O char 0} gives an isomorphism $\ol \fX^r_K\isora\thnonr$. When $r$ is an integer, the construction of the isomorphism is similar to the equal characteristic case (\ref{X int r}).

\begin{remark}\label{ker A to O}
From \eqref{I/I2 iso O char 0}, we notice that for any element $\wt x\in \ker(\cA_K\ra\OK)$, the class $(\ol{\wt x\ot 1})\ot 1$ vanishes in $(I_K/I_K^2)\ot_{\OK}F$. It is equivalent to say that $\wt x\ot 1\in I_K^2+\p I_K$.
\end{remark}

\subsection{}\label{th to X char 0}
Let $L$ be a finite Galois extension of $K$ of group $G$ and conductor $c>1$. Let $(g,\bmg)$ be a finite and flat morphism from $(\cA_K,j_K:\cA_K\ra \OK)$ to $(\cA_L,j_L:\cA_L\ra \OL)$  (\ref{present cartier}).
By \eqref{key iso char 0}, $\bmg$ induces a finite flat morphism $\bmg\ot\id:\cA_K\wdhat\ot_{\cO_{K_0}}\OK\ra \cA_{L}\wdhat\ot_{\cO_{K_0}}\OK$. Hence, for any rational number $r>0$, it gives a morphism of smooth $K$-affinoid varieties $X^r_L\ra X^r_K$ (\cite{as ii} 1.6) which induces morphisms $\fX^r_L\ra\fX^r_K$ and $\ol\fX^r_K\ra\ol\fX^r_L$. For any $\s\in G$, there is a morphism $\bmg_{\s}$ making the following diagram commutative (\ref{step 0 lemma} iii)
\begin{equation}\label{gs}
\xymatrix{\relax
\cA_{L}\ar[d]_{\bmg_{\s}}\ar[r]^{j_L}&\cO_{L}\ar[d]^{\s}\\
\cA_{L}\ar[r]^{j_L}&\cO_{L}}
\end{equation}
The pair $(\s,\bmg_{\s})$ induces automorphisms of $X^r_L$, $\fX^r_L$ and $\ol\fX^r_L$. Notice that, $\bmg_{\s}$ is not unique in general and may not be an $\cA_K$-homomorphism. Hence the automorphisms of $X^r_L$, $\fX^r_L$ and $\ol\fX^r_L$ induced by all possible $\bmg_{\s}$ may not be uniquely determined by $\s\in G$. Luckily, by (\cite{as ii} 2.13), the induced automorphism of $\ol\fX^c_L$ is $\ol\fX_K^c$-invariant and independent of the choice of $\bmg_{\s}$. Hence $\ol\fX^c_L\ra \ol\fX^c_K$ is a finite \'etale $G$-torsor (\cite{as ii} 1.16.2). The geometric monodromy action of $G_K$ on $\ol\fX^c_{L}$ commutes with the action of $G$. Let $\ol\fX^c_{L,0}$ be a connected component of $\ol\fX^c_{L}$. The stabilizers of $\ol\fX^c_{L,0}$ via these two actions are $G^c$ and $G^c_K$, respectively (\cite{as ii} 2.15.1). Then, we get an isomorphism $G^c\isora \aut(\ol\fX^c_{L,0}/\ol\fX^c_K)$ and a surjection $G^c_K\ra \aut(\ol\fX^c_{L,0}/\ol\fX^c_K)$ which imply that $G^c$ is commutative (cf. \cite{as ii} 2.15.1). Composing with $\ol\fX^r_K\isora\thnonr$, the \'etale covering $\ol\fX^c_{L,0}\ra \thnonr$ induces a surjective homomorphism (\cite{as ii} 2.15.1)
\begin{equation*}
\p^{\ab}_1(\thnonr)\ra \gr^cG_K\ra G^c.
\end{equation*}

\subsection{}
In the following of this section, we assume the finite Galois extension $L/K$ of type (II) and we take again the notation and assumptions of \ref{ram type ii notation} and \ref{ram type ii notation2}. Let $(g,\bmg)$ be a finite and flat morphism as in \ref{th to X char 0}. Let $\widetilde h$ be a lifting of $h\in \OL$ in $\cA_L$. Since $\cA_K$ is a Noetherian local ring, by \ref{step 0 lemma} (iii) and Nakayama's lemma, we have that $\cA_L$ is a finite free $\cA_K$-module of rank $\sharp G$ and that $\cA_L=\cA_K[\wt h]$. Let
\begin{equation*}
\wt f(T)=T^{p^n}+\wt a_{p^n-1}T_{p^n-1}+\cdots+\wt a_0\in\cA_K[T],
\end{equation*}
be a lifting of $f[T]\in\OK[T]$ such that $\wt h$ is a zero. We have an  isomorphism
\begin{equation}\label{AL=AKh}
\cA_K[T]/(\wt f(T))\isora \cA_L,\ \ \ T\ma\wt h.
\end{equation}
By \eqref{key iso char 0} and the proof of (\cite{as ii} 1.6), we have an isomorphism
\begin{equation}\label{cart1 char 0}
\cR_K^r\ot_{\cA_K\swdhat{\ot}_{\cO_{K_0}}\OK}(\cA_L\wdhat{\ot}_{\cO_{K_0}}\OK) \isora \cR_L^r.
\end{equation}
It induces, for any rational numbers $r>r'>0$, an isomorphism
\begin{equation*}
\cR^{r}_K\ot_{\cR^{r'}_K}\cR_{L}^{r'}\isora \cR^{r}_L,
\end{equation*}
which gives a Cartesian diagram of rigid spaces
\begin{equation}\label{tub morphism char 0}
\xymatrix{\relax
X^{r}_L\ar[d]\ar[r]&X_L\ar[d]\\
X^{r}_K\ar[r]&X_K}
\end{equation}
where $X_K=\bigcup_{r>0}X^r_K$ and $X_L=\bigcup_{r>0}X^r_L$. We put
\begin{equation*}
\wt\bmf(T)=T^{p^n}+(\wt a_{p^n-1}\ot 1)T^{p^n-1}+\cdots+(\wt a_0\ot 1)\in (\cA_{K}\wdhat{\ot}_{\cO_{K_0}}\OK)[T].
\end{equation*}
From \eqref{key iso char 0} and \eqref{cart1 char 0}, we have a surjection
\begin{equation*}
\t_L:\cR^r_K\<T\>\ra \cR^r_L,\ \ \ T\ma\wt h\ot 1,
\end{equation*}
 which induces an isomorphism that we denote abusively also by
\begin{equation}\label{key representation char 0}
\t_L:\cR^r_K\<T\>/\wt\bmf(T)\isora \cR^r_L.
\end{equation}
In other terms, we have a co-Cartesian diagram of homomorphisms of $\cR^r_K$-algebras
\begin{equation}\label{cocart char 0}
\xymatrix{\relax
\cR^r_L&\cR^r_K\<T\>\ar[l]_{\t_L}\\
\cR^r_K\ar[u]&\cR^r_K\<T\>\ar[u]_{\phi}\ar[l]_{\t_K},
}
\end{equation}
where $\phi(T)=\wt\bmf(T)$ and $\t_K(T)=0$. Hence, taking the union of the $K$-affinoid varieties associated to each of the $K$-affinoid algebras in \eqref{cocart char 0} for $r\in\Q_{\geqslant 0}$, we obtain a Cartesian diagram
\begin{equation}\label{tub explicit char 0}
\xymatrix{\relax
X_L\ar[d]\ar[r]^-(0.5){i_L}&X_K\ti D^1\ar[d]^{\wt\bff}\\
X_K\ar[r]^-(0.5){i_K}&X_K\ti D^1,}
\end{equation}
where $i_L$, $\wt\bff$ and $i_K$ are the morphisms induced by $\t_L$, $\phi$ and $\t_K$.

\subsection{}\label{ab char 0}
In the following, for any $0\leqslant i\leqslant p^n-1$, we denote by $\a_i$ the element $ a_i-\wt a_i\ot 1\in I_K$ and fix $\wt\p\in\cA_K$ a lifting of $\p\in\OK$. When the conductor $c>2$, for each $1\leqslant i\leqslant p^n-1$, $v(a_i)\geqslant 2$ (\ref{p^2}). Let $ a'_i= \p^{-2} a_i\in\OK$ and $\wt a'_i\in \cA_K$ a lifting of $a'_i$. Then we have $\wt a_i=\wt\p^2\wt a_i'+\wt y_i$, where $\wt y_i\in\ker(\cA_K\ra\OK)$. We denote by $\a'_i$ the element $a'_i-\wt a'_i\ot 1\in I_K$ and by $\b$ the element $\p-\wt\p\ot1\in I_K$. Then, we have
\begin{equation*}
\a_i=(a'_i-\a'_i)(2\p\b-\b^2)+\p^2\a'_i+\wt y_i\ot 1.
\end{equation*}
Since $\a'_i$, $\b\in I_K \sub\p^c\cR^c_{K,\OK}$ and $\wt y_i\ot 1\in I_K^2+\p I_K\sub\p^{c+1}\cR^c_{K,\OK}$ (\ref{ker A to O}), we have $\a_i\in \p^{c+1}\cR^{c}_{K,\OK}$.
When $c=2$, we have $p=2$, $\sharp G=2$, $\deg f=2$ and $\r(c)=1$. Let $\wt a''_1\in \cA_K$ be a lifting of $a''_1=\p^{-1}a_1$. We have $\a_1=\wt\p\wt a''_1+\wt z_1$, where $z_1\in\ker(\cA_K\ra\OK)$. We denote by $\a''_1$ the element $a''_1-\wt a''_1\ot 1\in I_K$. Then we have
\begin{equation*}
\a_1=(a''_1-\a''_1)\b+\p \a''_1+\wt z_1\ot1.
\end{equation*}
Since $\a''_1,\b\in\p^c\cR^c_{K,\OK}$ and $\wt z_1\ot 1\in I_K^2+\p I_K\sub\p^{c+1}\cR^c_{K,\OK}$, we have $\a_1\in \p^c\cR^{c}_{K,\OK}$, and $\ol{\a_1/\p^c}=\ol{a''_1\b/\p^c}\in \cR^{c}_{K,\OK}\big/\p\cR^{c}_{K,\OK}$.

Put
\begin{equation*}
\wt\bmf_0(T)=\sum_{0\leqslant i\leqslant p^n-1}(\a_i/\p^c)\cdot T^i\in \cR^c_{K,\OK}[T].
\end{equation*}
We have
\begin{equation*}
\wt\bmf(T)=f(T)-\sum_{0\leqslant i\leqslant p^n-1}\a_iT^i=f(T)-\p^{c} \wt\bmf_0(T).
\end{equation*}

In the following, we fix an embedding $L\ra\ol K$.

\begin{proposition}\label{component char 0}
The $K$-affinoid $X^{c}_L$ has $\sharp(G/G^{c})=p^{n-s}$ geometric connected components. Let $\s_1,...,\s_{p^{n-s}}$ be liftings of all the elements of $G/G^{c}$ in $G$. We have
\begin{equation*}
i_L(X^{c}_L)\sube\coprod_{1\leqslant j\leqslant p^{n-s}} X^{c}_K\ti (\s_j(h)+D^{1,(\r(c))})\sube X_K\ti D^1.
\end{equation*}
\end{proposition}
\begin{proof}
The proof is the same as in the equal characteristic case (\ref{component}).
\end{proof}
In the following, we denote by $\ol\fX^c_{L,0}$ the connected component of $\ol\fX^c_L$ corresponding to the connected component $X^c_{L,0}$ of $X^c_L$ containing $(0,...,0,h)\in X^c_K\ti D^1$ defined over $L$.

\begin{proposition}\label{hatO isogeny char 0}
There exists a canonical Cartesian diagram
\begin{equation}\label{cartesian diagram isogeny char 0}
\xymatrix{\relax
\ol\fX^c_{L,0}\ar[d]\ar[r]^{\nu}&\A^1_{\ol F}\ar[d]^{\ol {f_c}}\\
\thnonc\ar[r]^{\mu}&\A^1_{\ol F},}
\end{equation}
where $\ol{f_{c}}$ is defined in \eqref{f_g} and if $\xi$ is the canonical coordinate of $\A^1_{\ol F}$, we have
\begin{equation*}
\mu^*(\xi)=\lt\{\begin{array}{ll}\dr a_0\ot\p^{-c},&\text{if}\ \ c>2,\\
(a''_1h\dr\p+\dr a_0)\ot \p^{-2},&\text{if}\ \ c=2.
\end{array}\rt.
\end{equation*}
Moreover, for any $\s\in G^c$, the following diagram
\begin{equation}\label{G action diagram char 0}
\xymatrix{\relax
\ol \fX^c_{L,0}\ar[d]_{\s}\ar[r]^{\nu}&\A^1_{\ol F}\ar[d]^{d_{\s}}\\
\ol \fX^c_{L,0}\ar[r]^{\nu}&\A^1_{\ol F},}
\end{equation}
where $d_{\s}^{*}(\x)=\x-u_{\s}$ (\ref{u}), is commutative.
\end{proposition}

\begin{proof}
We consider the $K$-affinoid algebra $\cR^c_K$ (resp. $\cR^c_L$) as a sub-ring of the $L$-affinoid algebra $\cR^c_K \ot_K L$ (resp. $\cR^c_L \ot_K L$).
By \eqref{component char 0}, we have
\begin{equation*}
X^c_{L,0}=i^{-1}_L(X^c_K\ti(h+D^{1,(\r(c))}))\cap X^c_L.
\end{equation*}
Hence $X^c_{L,0}$ is presented by the $L$-affinoid algebra
\begin{equation}\label{connect comp over L char 0}
(\cR^{c}_L\ot_{K}L)\<T'\>/(\p^{\r(c)}T'+h-\wt h\ot 1).
\end{equation}
By the isomorphism \eqref{key representation char 0}, \eqref{connect comp over L char 0} is isomorphic to
\begin{equation*}
(\cR^{c}_K\ot_{K}L)\<T,T'\>/(\wt\bmf(T), \p^{\r(c)}T'+h-T),
\end{equation*}
which, after eliminating $T$ by the relation $\p^{\r(c)}T'+h-T=0$, is
\begin{equation}\label{connected comp char 0}
(\cR^{c}_K\ot_{K}L)\<T'\>/(\wt\bmf(\p^{\r(c)}T'+h)).
\end{equation}
In both cases, by \ref{isogeny f} and \ref{ab char 0}, we have
\begin{equation*}
\wt\bmf(\p^{\r(c)}T'+h)/\p^{c}\in \cR^{c}_{K,\OL}\<T'\>,
\end{equation*}
\begin{equation*}
\wt\bmf(\p^{\r(c)}T'+h)/\p^{c+1}\notin\cR^{c}_{K,\OL}\<T'\>.
\end{equation*}
Then the image of $\cR^{c}_{K,\OL}\<T'\>$ in \eqref{connected comp char 0} through the canonical surjective map
\begin{equation*}
(\cR^c_K\ot_KL)\<T'\>\ra(\cR^{c}_K\ot_{K}L)\<T'\>/(\wt\bmf(\p^{\r(c)}T'+h)),
\end{equation*}
is
\begin{equation}\label{oB char 0}
\cR^{c}_{K,\OL}\<T'\>/(\wt\bmf(\p^{\r(c)}T'+h)/\p^{c}).
\end{equation}
Extending scalars from $\OL$ to $\ol F$, we obtain the following $\ol F$-algebra:
\begin{itemize}
\item[(i)]
if $c>2$,
\begin{equation}\label{clos fiber >2 char 0}
   (\cR^{c}_{K,\OL}\ot_{\OL}\ol F)[T']/(\ol { f_{c}}(T')-\ol{\a_0/\p^{c}}).
\end{equation}
\item[(ii)]
if $c=2$,
\begin{equation}\label{clos fiber =2 char 0}
(\cR^{c}_{K,\OL}\ot_{\OL}\ol F)[T']/(\ol{f_2}(T')- \ol{(\a_0+a_1''h\b)/\p^{2}}).
\end{equation}
\end{itemize}
From the isomorphism \eqref{I/I2 iso O char 0} and the canonical exact sequence \eqref{hatO exact 0}, we know that when $c>2$ (resp. $c=2$), $\ol{\a_0/\p^c}$ (resp. $\ol{(\a_0+a_1''h\b)/\p^{2}}$) is a non-zero linear term in $\cR^{c}_{K,\OL}\ot_{\OL}\ol F$.
Hence \eqref{clos fiber >2 char 0} and \eqref{clos fiber =2 char 0} are all reduced.
Then, by (\cite{as i} 4.1),
\begin{equation*}
\spf(\cR^{c}_{K,\OL}\<T'\>/(\wt\bmf(\p^{\r(c)}T'+h)/\p^{c}))
\end{equation*}
is the normalized integral model of $X^c_{K,0}$ defined over $\OL$. Hence $\ol\fX^{c}_{L,0}$ is defined by the $\ol F$-algebra \eqref{clos fiber >2 char 0} (resp. \eqref{clos fiber =2 char 0}) when $c>2$ (resp. $c=2$). We put
\begin{equation*}
\nu:\ol\fX^c_{L,0}\ra\A^1_{\ol F}=\spec(\ol F[\x]),\ \ \ \nu^*(\x)=T'.
\end{equation*}
It follows form the isomorphism $\ol\fX^{c}_K\ra \thnonc$ that \eqref{cartesian diagram isogeny char 0} is Cartesian.

For any $\s\in G^c$, let $y_{\s}(x)=b_rx^r +\cdots +b_0\in \OK[x]$ be a polynomial, such that $y_{\s}(h)=(h-\s(h))/\p^{\r(c)}\in \OL$. We denote by $\wt y_{\s}(x)=\wt b_rx^r +\cdots +\wt b_0$ a lifting of $y_{\s}(x)$ in $\cA_K[x]$ and by $\wt{\mathbbm y}(x)$ the polynomial
\begin{equation*}
\wt {\mathbbm y}(x)=(\wt b_r\ot 1)x^r +\cdots +(\wt b_0\ot 1)\in (\cA_K\wdhat\ot_{\cO_{K_0}}\OK)[x].
\end{equation*}
Let $\bmg_{\s}:\cA_L\ra\cA_L$ be a homomorphism as in $\eqref{gs}$. We denote by $\bfg_{\s}$ the induced morphism of $\bmg_{\s}$ on \eqref{oB char 0}. By \eqref{AL=AKh}, we have
\begin{equation*}
\ker(\cA_L\ra\OL)=\bp_{i=0}^{p^n-1}\ker(\cA_K\ra\OK) \wt h^i.
\end{equation*}
Hence, we have $\bmg_{\s}(\wt h)=\wt h- \wt \p^{\r(c)}\wt y_{\s}(\wt h)+\varepsilon(\wt h)$, where $\varepsilon$ is a polynomial with coefficients in $\ker(\cA_K\ra \OK)$. Then, we have
\begin{equation*}
\bfg_{\s}(T')=T'-\wt{\mathbbm y}_{\s}(\p^{\r(c)}T'+h)+\D(T'),
\end{equation*}
where
\begin{equation*}
\D(T')=-((\wt\p^{\r(c)}\ot1-\p^{\r(c)})/\p^{\r(c)})\wt{\mathbbm y}_{\s}(\p^{\r(c)}T'+h)+\wt\varepsilon(\p^{\r(c)}T'+h)/\p^{\r(c)},
\end{equation*}
and $\wt\varepsilon$ is a polynomials with coefficients in $J=\{\wt x\ot 1\in\cA\wdhat\ot_{\cO_{K_0}}\OK;\wt x\in \ker(\cA_K\ra \OK)\}$. Since $J\sube \p^{c+1}\cR_{K,\OK}$ (\ref{ker A to O}),  $\wt\p^{\r(c)}\ot1-\p^{\r(c)}\in\p^{c}\cR_{K,\OK}$ and $c>\r(c)$, it is easy to see that the reduction of $\D(T')$ to $\ol\fX^c_{L,0}$ is zero. For any $0\leqslant j\leqslant r$, we have $\wt b_j\ot 1- b_j\in \p^c\cR^c_{K,\OK}$. Then
\begin{equation*}
\ol {\wt{\mathbbm y}_{\s}(\p^{\r(c)}T'+h)}=\ol {\wt y_{\s}(\p^{\r(c)}T'+h)}=\ol{\wt y_{\s}(h)}=u_{\s}.
\end{equation*}
Hence, by (\cite{as ii} 2.13), the diagram \eqref{G action diagram char 0} is commutative.
\end{proof}

\section{The refined Swan conductor of an extension of type (II)}
\subsection{}
In this section, we assume either that $K$ has characteristic $p$ or that it has characteristic $0$ and that $p$ is not a uniformizer of $K$. Let $L$ be a finitely generated extension of $K$ of type (II) and we take again the notation and assumptions of \ref{ram type ii notation}, \ref{ram type ii notation2}, \ref{thnong isogeny} and \ref{hatO isogeny char 0}.

\begin{proposition}\label{isogeny xi}
The fibre product $\ol\fX^{c}_{L,0}\ti_{\thnonc}\xic$ \eqref{xi} is a connected affine scheme.
\end{proposition}
\begin{proof}
The image of $\dr a_0\ot 1$ and $(a''_1h\dr \p+\dr a_0)\ot 1$ by the canonical map from $\what\O^1_{\OK/F_0}\ot_{\OK}\ol F$ (resp. $\what\O^1_{\OK/\cO_{K_0}}\ot_{\OK}\ol F $) to $\O^1_F\ot_F \ol F$ is $\dr \bar a_0\ot 1$, which is a non-zero element. So we have a Cartesian diagram
\begin{equation}\label{isogeny xi1}
\xymatrix{
\ol\fX^c_{L,0}\ti_{\thnonc}\xic\ar[d]\ar[r]&\A^1_{\ol F}\ar[d]^{\ol{f_c}}\\
\xic\ar[r]^{\mu'}&\A^1_{\ol F}}
\end{equation}
where $\mu'^*(\xi)=\dr \bar a_0\ot \p^{-c}$.
Since $\dr \bar a_0\ot \p^{-c}$ is a non-zero linear term in the affine space $\xic$, $\ol\fX^{c}_{L,0}\ti_{\thnonc}\xic$ is connected.
\end{proof}

\subsection{}\label{proof of proposition}
{\em Proof of \ref{thlog factor xi}}.
By (\cite{as ii} 5.13), both in the equal and unequal characteristic case, we have a commutative diagram
\begin{equation}\label{as ii diagram}
\xymatrix{\relax
\p^{\ab}_1(\thlogc)\ar[r]\ar@{->>}[d]_{\g_1}&\p^{\ab}_1(\thnonc)\ar@{->>}[d]^{\g_2}\\
G^{c}_{\log}\ar@{=}[r]&G^{c}.}
\end{equation}
The surjection $\g_1$ factors through $\p^{\alg}_1(\thlogc)$ (\ref{isogeny}). By \ref{thnong isogeny} and \ref{hatO isogeny char 0}, $\g_2$ also factors through $\p^{\alg}_1(\thnonc)$. Combining \eqref{as ii diagram} and the following canonical commutative diagram
\begin{equation*}
\xymatrix{\relax
\p^{\ab}_1(\thlogc)\ar[r]\ar@{->>}[d]&\p^{\ab}_1(\xic)\ar@{->>}[d]\ar[r]&\p^{\ab}_1(\thnonc)\ar@{->>}[d]\\
\p^{\alg}_1(\thlogc)\ar[r]&\p^{\alg}_1(\xic)\ar[r]&\p^{\alg}_1(\thnonc),}
\end{equation*}
we obtain that
\begin{equation}\label {key diagram}
\xymatrix{\relax
\p^{\alg}_1(\thlogc)\ar[r]\ar@{->>}[d]&\p^{\alg}_1(\xic)\ar[r]&\p^{\alg}_1(\thnonc)\ar@{->>}[d]\\
G^{c}_{\log}\ar@{=}[rr]& &G^{c}}
\end{equation}
is commutative. The composition of morphisms $\p^{\alg}_1(\xic)\ra\p^{\alg}_1(\thnonc)\ra G^c$ corresponds the isogeny $\ol\fX^c_{L,0}\ti_{\thnonc}\xic\ra\xic$ (cf. \eqref{isogeny xi1}). Hence, by \eqref{key diagram}, we have a commutative diagram
\begin{equation}\label{diagram refined swan}
\xymatrix{\relax
 &\Hom(\p^{\alg}_1(\xic),\F_p)\ar[r]\ar[d]&\O^1_F\ot_{F}\m^{-c}_{\ol K}/\m^{-c+}_{\ol K}\ar[d]\\
\Hom(G^{c},\F_p)\ar[r]\ar[ur]&\Hom(\p^{\alg}_1(\thlogc),\F_p)\ar[r]&\O^1_F(\log)\ot_{F}\m^{-c}_{\ol K}/\m^{-c+}_{\ol K},}
\end{equation}
which conclude \ref{thlog factor xi}.

\subsection{}\label{proof of theorem}
{\em Proof of \ref{theorem rsw}.}
Since the surjection $\p^{\alg}_1(\xic)\ra G^c$ is obtained by pulling-back the isogeny $\ol{f_c}:\A_{\ol F}^1\ra\A_{\ol F}^1$ by $\mu'$ (cf. \eqref{isogeny xi1}),
which is an \'etale $G^c$-torsor with the action of $G^c$ given by $d_{\s}$ for $\s\in G^c$ \eqref{G action diagram}, \eqref{G action diagram char 0}.
With notation in \ref{notation for rsw}, we denote by $ \tilde f_{c,\c}(\xi)$ the polynomial
\begin{equation*}
\tilde f_{c,\c}(\xi)=\lt(\prod_{\s\in G-G^{c}}u_{\s}\rt)(\xi^p-\bar f_{c,\c}^{p-1}(u_{\t})\xi)\in \ol F[\xi].
\end{equation*}
Observe that $\tilde f_{c,\c}(\bar f_{c,\c}(\xi))=\ol {f_{c}}(\xi)$, hence the isogeny $\ol{f_c}$ is the composition of two isogenies
\begin{equation*}
\A^1_{\ol F}\xra{\bar f_{c,\c}}\A^1_{\ol F}\xra{\tilde f_{c,\c}}\A^1_{\ol F}.
\end{equation*}
For any $\s\in \ker \c$, $\bar f_{c,\c}^*(\xi-u_{\s})=\bar f_{c,\c}^*(\xi)$, i.e. $\bar f_{c,\c}d_{\s}=\bar f_{c,\c}$. Hence the isogeny $\tilde f_{c,\c}:\A^1_{\ol F}\ra \A^1_{\ol F}$ is an \'etale $(G^c/\ker\c)$-torsor. Then, the surjection $\p^{\alg}_1(\xic)\ra G^c\xra{\c}\F_p$ corresponds to the pull-back of $\tilde f_{c,\c}$ by $\mu'$ and the $\F_p$-action on this torsor is given by $1^*:\xi\ma \xi-\bar f_{c,\c}(u_{\t})$. We have the following Cartesian diagram
\begin{equation}\label{serre}
\xymatrix{\relax
  \F_p\ar[d]_{\id}\ar[r]^{\phi} & \A^1_{\ol F}\ar[d]_{\l_2}\ar[r]^{\tilde f_{c,\c}} & \A^1_{\ol F}\ar[d]^{\l_1} \\
  \F_p\ar[r] & \A^1_{\ol F}\ar[r]^{\mathrm{L}} & \A^1_{\ol F}}
\end{equation}
where $\mathrm{L}$ denotes the Lang's isogeny defined by $\mathrm{L}^*(\xi)=\xi^p-\xi$. The morphisms $\l_1$, $\l_2$ and $\phi$ are given as follows
\begin{equation*}
\l_1^*(\xi)=-\xi\bigg/\bigg(\prod\limits_{\s\in G-G^{c}}u_{\s}\bigg)\bar f_{c,\c}^p(u_{\t}),
\end{equation*}
\begin{equation*}
\l_2^*(\xi)=-\xi/\bar f_{c,\c}(u_{\t}),
\end{equation*}
\begin{equation*}
\phi(1)=-\bar f_{c,\c}(u_{\t}).
\end{equation*}
The sign is chosen in order that, for any $\s\in G^c$, the translation by $\phi(\c(\s))$ is induced by $d_{\s}$. Consequently, $\p^{\alg}_1(\xic)\ra G^c\xra{\c}\F_p$ corresponds to the pull-back of $\mathrm{L}$ by $\l_1\mu'$. Hence the image of $\c\in\Hom(G^{c},\F_p)$ in $\O^1_F\ot_{F}\m^{-c}_{\ol K}/\m^{-c+}_{\ol K}$ \eqref{diagram refined swan} is
\begin{equation*}
-\dr \bar a_0\ot\f{\p^{-c}}{\lt(\prod_{\s\in G-G^{c}}u_{\s}\rt)\bar f_{c,\c}^p(u_{\t})}\in \O^1_F\ot_{F}\m^{-c}_{\ol K}/\m^{-c+}_{\ol K}.
\end{equation*}
Then the theorem follows from \eqref{diagram refined swan}.

\section{Comparison of Kato's and Abbes-Saito's characteristic cycles}
\subsection{}
In this section, let $L$ be a finite Galois extension of $K$ of type (II) and we take again the notation and assumptions of \ref{ram type ii notation} and \ref{ram type ii notation2}. Let $C$ be an algebraically closed field of characteristic zero. We fix a non-trivial character $\psi_0:\F_p\ra C^{\ti}$. Any character $\c:G^c\ra C^{\ti}$ factors uniquely through $G^c\ra\F_p\xra{\psi_0} C^{\ti}$. We denote by $\bar\c$ the induced character $G^c\ra \F_p$.

\begin{proposition}\label{rank 1 theorem}
Let $\c: G\ra C^{\ti}$ be a character of $G$ such that its restriction to $G^c$ is non-trivial. Let $\t\in G^c$ be a lifting of $1\in \F_p$ in $G^c$ through $\bar\c:G^c\ra\F_p$. Then Kato's swan conductor with differential values $\sw_{\psi_0(1)}(\c)$ is given by \eqref{u}, \eqref{notation for rsw}
\begin{equation*}
\sw_{\psi_0(1)}(\c)=[\p^c]+[-\bar f^p_{c,\bar\c}(u_{\t})]+\sum_{\s\in G-G^c}[u_{\s}]-[\dr\bar a_0]\in \skl.
\end{equation*}
\end{proposition}
\begin{proof}
By definition \eqref{swan}, we have
\begin{eqnarray}
\sw_{\psi_0(1)}(\c)&=&\sum_{\s\in G-\{1\}}([h-\s(h)]-[\dr \bar h])\ot(1-\c(\s))+\sum_{r\in \F_p^{\ti}}[r]\ot\psi_0(r)\label{sw 1 dim}\\
&=&\sum_{\s\in G^c-\{1\}}[h-\s(h)]\ot(1-\c(\s))+\sum_{r\in \F_p^{\ti}}[r]\ot\psi_0(r)\label{G^c}+ \nonumber\\
& &\sum_{\s\in G-G^c}[h-\s(h)]-\sum_{\s\in G-G^c}[h-\s(h)]\ot\c(\s)- \nonumber\\
& &\sum_{\s\in G-\{1\}}[\dr \bar h]\ot(1-\c(\s)). \nonumber
\end{eqnarray}
Choose an $\F_p$-basis $\t_1=\t,\t_2,...,\t_s$ of $G^{c}$ such that $\bar\c(\t_1)=1\in\F_p$ and, that for any $2\leqslant j\leqslant s$, $\bar\c(\t_j)=0$. Then, by \ref{isogeny lemma}, we have
\begin{eqnarray}
\,&\,&\sum_{\s\in G^c-\{1\}}[h-\s(h)]\ot(1-\c(\s))+\sum_{r\in \F_p^{\ti}}[r]\ot\psi_0(r)\label{G^c}\\
&=&[\p^{\r(c)\sharp G^c}]+\sum_{\{j_1,...,j_s\}\in\F^s_p-\{0\}}[j_1u_{\t_1}+\cdots+j_s u_{\t_s}]\ot(1-\psi_0(j_1))+\sum_{r\in \F_p^{\ti}}[r]\ot\psi_0(r)\nonumber\\
&=&[\p^{\r(c)\sharp G^c}]+\sum_{r\in\F_p^{\ti}}[\bar f_{c,\bar\c}(r u_{\t_1})]\ot(1-\psi_0(r))+\sum_{r\in \F_p^{\ti}}[r]\ot\psi_0(r)\nonumber\\
&=&[\p^{\r(c)\sharp G^c}]+\sum_{r\in\F_p^{\ti}}([\bar f_{c,\bar\c}(u_{\t_1})]+[r])\ot(1-\psi_0(r))+\sum_{r\in \F_p^{\ti}}[r]\ot\psi_0(r)\nonumber\\
&=&[\p^{\r(c)\sharp G^c}]+\sum_{r\in\F_p^{\ti}}[\bar f_{c,\bar\c}(u_{\t_1})]\ot(1-\psi_0(r))+\sum_{r\in \F_p^{\ti}}[r]\nonumber\\
&=&[\p^{\r(c)\sharp G^c}]+[-\bar f_{c,\bar\c}^p(u_{\t_1})]\in\slk.\nonumber
\end{eqnarray}
Let $\s_1=1,\s_2...,\s_{p^{n-s}}$ be a lifting of all the element of $G/G^c$ in $G$ and denote by $J$ the set $\{\s_2,...,\s_{p^{n-s}}\}$. Observe that for any $\varsigma\in J$ and $\s\in G^c$, we have
\begin{equation*}
[h-\varsigma\s(h)]=[h-\varsigma(h)+\varsigma(h-\s(h))]=[h-\varsigma(h)].
\end{equation*}
Hence
\begin{eqnarray}
\sum_{\s\in G-G^c}[h-\s(h)]\ot\c(\s)&=&\sum_{\varsigma\in J}\sum_{\s\in G^c}[h-\varsigma\s(h)]\ot\c(\varsigma\s)\label{G-G^c}\\
&=&\sum_{\varsigma\in J}\sum_{\s\in G^c}[h-\varsigma(h)]\ot\c(\varsigma)\c(\s)=0.\nonumber
\end{eqnarray}
Moreover, by the isomorphism \eqref{vo}, we have
\begin{equation}\label{dh da}
\sum_{\s\in G^c-\{1\}}[\dr \bar h]\ot(1-\c(\s))=\sharp G[\dr\bar h]=[\dr \bar a_0]\in \skl.
\end{equation}
Hence, combining \eqref{sw 1 dim}, \eqref{G^c}, \eqref{G-G^c} and \eqref{dh da}, we obtain that
\begin{eqnarray}
\sw_{\psi_0(1)}(M)&=&[\p^{\r(c)\sharp G^c}]+[-\bar f_{c,\c}^p(u_{\t_1})]+\sum_{\s\in G-G^c}[h-\s(h)]\ot1-\sharp G[\dr\bar h]\nonumber\\
&=&[\p^c]+[-\bar f^p_{c,\bar\c}(u_{\t})]+\sum_{\s\in G-G^c}[u_{\s}]-[\dr\bar a_0].\nonumber
\end{eqnarray}
\end{proof}

\begin{lemma}\label{reduce to rank 1 key lemma}
Let $M$ be a finite dimensional $C$-vector space with an irreducible linear action of $G$. Then, there exists a subgroup $H$ of $G$ satisfying $G^c\sube H$ and a 1-dimensional representation $\th$ of $H$, such that $M=\ind_H^G\th$.
\end{lemma}
\begin{proof}
Since $M$ is irreducible and $G$ is nilpotent (hence super-solvable), there exist a subgroup $H$ of $G$ and a 1-dimensional representation $\th$ of $H$, such that $M=\ind_H^G\th$ (\cite{serre gr} 8.5 Th. 16). Let $\res^G_{G^c}M=\bp_i M_i$ be the canonical decomposition of $\res^G_{G^c}M$ into isotypic $G^c$-representations (cf. \cite{serre gr} 2.6). Since $G^c$ is contained in the center of $G$, any $\s\in G$ defines an automorphism of the $G^c$-representation $\res^G_{G^c}M$. In particular, for any $i$, $\s$ induces an automorphism of $M_i$. On the other hand, since $M$ is irreducible, $G$ permutes transitively the $M_i$'s. Hence $\res^G_{G^c}M$ is isotypic. By (\cite{serre gr}  7.3 Prop. 22), we have
\begin{equation}\label{ind res}
\res_{G^c}^GM=\res_{G^c}^G\ind^G_H\th=\bp_{H\backslash G/G^c}\ind^{G^c}_{H\cap G^c}\res^H_{H\cap G^c}\th.
\end{equation}
We notice that, if $H\cap G^c\neq G^c$, since $G^c=H\cap G^c\op G^c/H\cap G^c$, $\ind^{G^c}_{H\cap G^c}\res^H_{H\cap G^c}\th$ is isomorphic to the tensor of the regular representation of $G^c/H\cap G^c$ with $\res^H_{H\cap G^c}\th$ which is not isotypic.
\end{proof}

\begin{theorem}\label{mainkey}
Assume that $p$ is not a uniformizer of $K$. Let $M$ be a finite dimensional $C$-vector space with a linear action of $G$. Then,
\begin{equation}\label{cckcc}
\cc_{\psi_0}(M)=\kcc_{\psi_0(1)}(M).
\end{equation}
\end{theorem}
\begin{proof}
From the definitions, we may assume that $M$ is irreducible. We denote by $c_0$ the unique slope of $M$. By definitions and \ref{scdv quotient}, both sides of \eqref{cckcc} will not change if replacing $G$ by $G/G^{c_0+}$. Hence we may assume further that the unique slope of $M$ is equal to $c$. By \ref{reduce to rank 1 key lemma}, $M=\ind_H^G\th$ where $H$ is a subgroup of $G$ containing $G^c$ and $\th$ is a character of $H$. Since the slope of $M$ is $c$, the restriction of $\th$ to $G^c$ is non-trivial \eqref{ind res}. We notice that $[G:H]=\dim_CM$. Choose an $\F_p$-basis $\t_1,...,\t_s$ of $G^{c}$ such that $\bar\th(\t_1)=1\in\F_p$ and, for any $2\leqslant j\leqslant s$, $\bar\th(\t_j)=0$. Let $c'=\r(c)+\sum_{\s\in H-\{1\}}v(h-\s(h))$. Since $L/L^H$ is still of type (II), we obtain that the conductor of $L/L^H$ is $c'$, that $H^{c'}=G^c$ and, denoting by $\r'$ the Herbrand function of $L/L^H$, that $\r'(c')=\r(c)$. Using (\ref{rank 1 theorem}) for the group $H$ and the representation $\th$, we have
\begin{equation}\label{sw th}
\sw_{\psi_0(1)}(\th)=[\p^{c'}]+[-\bar f^p_{c,\bar\th}(u_{\t_1})]+\sum_{\s\in H-H^{c'}}[u_{\s}]-\sharp H[\dr \bar h].
\end{equation}
 Meanwhile, we have
\begin{equation}\label{D-D}
-\sum_{\s\in G-H} ([\dr \bar h]-[h-\s(h)])=(\sharp H-\sharp G)[\dr \bar h]+[\p^{c-c'}]+\sum_{\s\in G-H}u_{\s}.
\end{equation}

Hence, combining \eqref{sw th}, \eqref{D-D} and the induction formula for Kato's swan conductors (\ref{indsw good}), we have
\begin{eqnarray}
\sw_{\psi_0(1)}(M)&=&[G:H]\lt(\sw_{\psi_0(1)}(\th)-\sum_{\s\in G-H} ([\dr \bar h]-[h-\s(h)]\rt)\nonumber\\
&=&[G:H]\lt( [\p^c]+[-\bar f^p_{c,\bar\th}(u_{\t_1})]-[\dr\bar a_0]+\sum_{\s\in G-G^c}[u_{\s}]\rt).\nonumber
\end{eqnarray}
Hence Kato's characteristic cycle $\kcc_{\psi_0(1)}(M)$ is given by
\begin{equation*}
\kcc_{\psi_0(1)}(M)=\f{(-\dr \bar a_0)^{\ot [G:H]}}{\big(\lt(\prod_{\s\in G-G^c}u_{\s}\rt)\bar f^p_{c,\bar\th}(u_{\t_1})\big)^{[G:H]}}\in (\O^1_F)^{\ot[G:H]}.
\end{equation*}
On the other hand, $\res^G_{G^c} M=\bp_{G/H}\res^H_{G^c} \th$ \eqref{ind res}. Hence the Abbes-Saito's characteristic cycle $\cc_{\psi_0}(M)$ is given by
\begin{equation}\label{cc ultimate}
\cc_{\psi_0}(M)=\lt(\rsw(\res^H_{G^c}(\th))\ot\p^c\rt)^{[G:H]}=\f{(-\dr \bar a_0)^{\ot [G:H]}}{\big(\lt(\prod_{\s\in G-G^c}u_{\s}\rt)\bar f^p_{c,\bar\th}(u_{\t_1})\big)^{[G:H]}}\in (\O^1_F\ot_{F}\ol F)^{\ot[G:H]}.
\end{equation}
So, we have $\cc_{\psi_0}(M)=\kcc_{\psi_0(1)}(M)$.
\end{proof}

\begin{corollary}\label{hasse arf cc}
Assume that $p$ is not a uniformizer of $K$. Let $M$ be a finite dimensional $\L$-vector space with a linear action of $G$ and $r=\dim_{\L}M/M^{(0)}$. Then, we have
\begin{equation*}
\cc_{\psi_0}(M)\in (\O^1_F)^r\sube(\O^1_F\ot_F \ol F)^r
\end{equation*}
\end{corollary}
It is a Hasse-Arf type result for Abbes-Saito characteristic cycle. We should mention that T.~Saito (\cite{wrcb} 3.10) and L. Xiao \cite{xiao} proved independently analogue results for smooth varieties of any dimension over perfect fields.

\begin{corollary}\label{indcc}
Assume that $p$ is not a uniformizer of $K$. Let $H$ be a sub-group of $G$, and $N$ a finite dimensional $C$-linear representation of $H$. We denote by $r$ the dimension of $N$ and by $r'$ the dimension of $N^{(0)}$. Then, we have
\begin{equation}\label{ind for cc}
\cc_{\psi_0}(\ind^G_HN)=\cc_{\psi_0}(N)^{\ot[G:H]}\ot\f{(\dr \bar a_0)^{\ot([G:H]-1)}}{\lt(\prod_{\s\in G-H}u_{\s}\rt)^{[G:H]}}\in(\O^1_F)^{\ot ([G:H]r-r')}.
\end{equation}
\end{corollary}
Indeed, \eqref{ind for cc} follows from the induction formula for Kato's swan conductor with differential values \eqref{indsw good} and \ref{mainkey}.

\begin{remark}\label{remark general equal}
Assume that $p$ is not a uniformizer of $K$. Let $L'$ be a finite Galois extension of $K$  of group $G'$ which contains a sub-extension $K'$ of $K$ such that $K'/K$ is unramified and $L'/K'$ is of type (II). We denote by $P'$ the Galois group of the extension $L'/K'$ and by $F'$ the residue field of $\cO_{K'}$. Let $\L$ be a finite field of characteristic $\ell\neq p$ which contains a primitive $(\sharp P')$-th root of unity and let $N$ be a $\L$-vector space of finite dimension with a linear-$G'$ action. We fix a non-trivial character $\psi:\F_p\ra \L^{\ti}$. By \ref{ext swan 1} and \ref{ext swan 2}, we can still define $\kcc_{\psi(1)}(N)\in(\O^1_F)^{\ot r}$, where $r=\dim_{\L}N/N^{(0)}$. On the other hand, the wild inertia subgroup $P$ of $G_K$ acts on $N$ through $P'$, we can define $\cc_{\psi}(N)$ (\ref{cc}). By (\cite{saito cc} 1.22) and (\cite{as iii} 3.1), we have
\begin{equation}\label{1}
\cc_{\psi}(\res^{G'}_{P'}N)=\cc_{\psi}(N)\in(\O^1_{F}(\log)\ot_{F}\ol F)^{\ot r}
\end{equation}
 through the canonical isomorphism $\O^1_F(\log)\ot_F F'\isora\O^1_{F'}(\log)$.
Moreover, let $\L'$ be the algebraic closure of the fraction field of the ring of Witt vectors $W(\L)$, $N'$ a pre-image of the class of $\res^{G'}_{P'}N$ in the Grothendieck ring $R_{\L'}(P')$ (\cite{serre gr} 16.1 Th. 33) and $\psi':\F_p\ra \L'^{\ti}$ the unique lifting of $\psi:\F_p\ra \L^{\ti}$. By \ref{slope center decom p to 0}, we deduce that
\begin{equation}\label{2}
\cc_{\psi'}(N')=\cc_{\psi}(\res^{G'}_{P'}N).
\end{equation}
From \ref{mainkey}, we have
\begin{equation}\label{3}
\cc_{\psi'}(N')=\kcc_{\psi(1)}(N).
\end{equation}
By \eqref{1}, \eqref{2} and \eqref{3}, we conclude that
\begin{equation}\label{general equal}
\cc_{\psi}(N)=\kcc_{\psi(1)}(N)\in(\O^1_F)^{\ot r}.
\end{equation}
\end{remark}

\section{Nearby cycles of $\ell$-sheaves on relative curves}
\subsection{}
In this section, we denote by $S=\spec(R)$ an excellent strictly henselian trait. Assume that the residue field of $R$ has characteristic $p$ and that $p$ is not a uniformizer of $R$. We denote by $s$ (resp. $\eta$, resp. $\bar\eta$) the closed point (resp. generic point, a geometric generic point) of $S$. A finite covering of $(S,\eta,s)$ stands for a trait $(S',\eta',s')$ equipped with a finite morphism $S'\ra S$. Let $\L$ be a finite field of characteristic $\ell\neq p$ and fix a non-trivial character $\psi_0:\F_p\ra \L^{\ti}$.

\subsection{}
We define a category $\sC_S$ as follows. An object of $\sC_S$ is a normal affine $S$-scheme $H$ for which there exist an $S$-scheme of finite type and a closed point $x$ of $X_s$, such that $X-\{x\}$ is smooth over $S$ and $H$ is $S$-isomorphic to the henselization of $X$ at $x$. A morphism between two objects of $\sC_S$ is a generically \'etale finite morphism of $S$-schemes. Let $(S',\eta',s')$ be a finite covering of $(S,\eta,s)$. Then for any object $H$ of $\sC_S$, $H\ti_SS'$ is an object of $\sC_{S'}$ (\cite{kato vc} 5.4).

\subsection{}
Let $H$ be an object of $\sC_S$. We denote by $P(H)$ the set of height 1 points of $H$, by
\begin{equation*}
P_s(H)=P(H)\cap H_s,\ \ \ P_{\eta}(H)=P(H)\cap H_{\eta}.
\end{equation*}
We have (\cite{kato vc} 5.2, \cite{as ft} A.6):
\begin{itemize}
\item[(i)]
$H_{\eta}$ is geometrically regular over $\eta$ and for any $\fp\in P_{\eta}(H)$, the residue field $\kappa(\fp)$ of $H$ at $\fp$ is a finite extension of the fraction field $K(S)$ of $S$.
\item[(ii)]
$H_s$ is a reduced henselian noetherian local scheme over $s$ of dimension 1, hence $P_s(H)$ is a finite set.
\end{itemize}
We denote by $\widetilde{H}_s$ the normalization of $H_s$, which is a finite union of strictly henselian traits. We put
\begin{equation*}
\d(H)=\dim_k(\sO_{\widetilde{H}_s}/\sO_{H_s}).
\end{equation*}

\subsection{}(\cite{as ft} A.7, A.8).\label{triple}
Let $H$ be an object of $\sC_S$, $U$ a non-empty open sub-scheme of $H_{\eta}$ and $\sF$ a locally constant constructible \'etale sheaf of ${\L}$-modules over $U$. For a triple $(H,U,\sF)$ and a finite covering $(S',\eta',s')$ of $(S,\eta,s)$, we denote by $(H,U,\sF)_{S'}$ the triple $(H',U',\sF')$ where $H'=H\ot_SS'$, $U'$ is the inverse image of $U$ in $H'$ and $\sF'$ is the inverse image of $\sF$ on $U'$.
We call the triple $(H,U,\sF)$ {\it stable} if there is an \'etale connected Galois covering $\widetilde U$ of $U$ such that
\begin{itemize}
\item[(i)]
The pull-back of $\sF$ to $\widetilde U$ is constant.
\item[(ii)]
The normalization $\widetilde H$ of $H$ in $\widetilde U$ belongs to $\sC_S$ and the residue field of $\widetilde H$ at all points of $\widetilde H_{\eta}-\widetilde U_{\eta}$ are finite separable extensions of $\k(\eta)$.
\end{itemize}

\begin{proposition}[\cite{kato vc} 6.3]\label{be stable}
Let $(H,U,\sF)$ be a triple as \ref{triple}.
\begin{itemize}
\item[(i)]
If $(H,U,\sF)$ is stable, $(H,U,\sF)_{S'}$ is stable for any finite covering $S'$ of $S$.
\item[(ii)]
For any triple $(H,U,\sF)$, there exist a finite covering $(S',\eta',s')$ of $(S,\eta,s)$ such that $(H,U,\sF)_{S'}$ is stable.
\end{itemize}
\end{proposition}
Proposition (i) follows form (\cite{kato vc} 5.4) and proposition (ii) follows form \cite{epp}.

\subsection{}
Let $(H,U,\sF)$ be a stable triple. For $\fp\in P(H)$, we denote by $\what{\cO}_{H,\fp}$ the completion of the local ring of $H$ at $\fp$ and by $\kappa(\fp)$ its residue field. For $\fp\in P_s(H)$, we denote by $\widetilde{H}_{s,\fp}$ the integral closure of $H_s$ in $\kappa(\fp)$, which is a strictly henselian trait. Let $\ord_{s,\fp}$ be the valuation of $\kappa(\fp)$ associated to $\widetilde{H}_{s,\fp}$ normalized by $\ord_{s,\fp}(\kappa(\fp)^{\ti})=\Z$. We denote also by $\ord_{s,\fp}:\O^1_{\kappa(\fp)}-\{0\}\ra \Z$ the valuation defined by $\ord_{s,\fp}(\a\dr\b)=\ord_{s,\fp}(\a)$, if $\a,\b\in\kappa(\fp)^{\ti}$ and $\ord_{s,\fp}(\b)=1$. It can be canonically extended, for any integer $r>0$, to $(\O^1_{\kappa(\fp)})^{\ot r}-\{0\}$. Following (\cite{sga7i} XVI, \cite{lau} and \cite{kato vc} 6.4), we call {\it total dimension} of $\sF$ at a point $\fp\in P(H)$, and denote by $\dimtot_{\fp}(\sF)$ the integer defined as follows:
\begin{itemize}
\item[(i)]
For $\fp\in P_{\eta}(H)$, we put
\begin{equation*}
\dimtot_{\fp}(\sF)=[\kappa(\fp):\k(\eta)](\sw_{\fp}(\sF)+\rank(\sF)),
\end{equation*}
where $\sw_{\fp}(\sF)$ is the Swan conductor of the pull-back of $\sF$ over $\spec(\what{\cO}_{H,\fp})\ti_HU$.
\item[(ii)]
For $\fp\in P_s(H)$, we denote by $K_{\fp}$ the fraction field of $\what {\cO}_{H,\fp}$. Since the triple $(H,U,\sF)$ is stable, there exists a finite Galois extension $L_{\fp}$ of $K_{\fp}$ of ramification index one, such that the representation $\sF_{\fp}$ of $\gal(K^{\mathrm{sep}}_{\fp}/K_{\fp})$ defined by $\sF$ factors through the quotient $\gal(L_{\fp}/K_{\fp})$. Notice that $L_{\fp}/K_{\fp}$ factors through a field $K'_{\fp}$ such that $K'_{\fp}/K_{\fp}$ is unramified and $L_{\fp}/K'_{\fp}$ is of type (II) (\ref{type}). Fixing a uniformizer $\p$ of $R$ (also a uniformizer of $K_{\fp}$), we have $\cc_{\psi_0}(\sF_{\fp})\in (\O^1_{\k(\fp)})^m$ (cf. \ref{remark general equal}).
We denote by $\ol{\sF}_{\fp}$ the restriction to $\spec(\kappa(\fp))$ of the direct image of $\sF$ under $\spec (K_{\fp})\ra \spec(\what {\cO}_{H,\fp})$ and by $\dimtot_{s,{\fp}}(\ol{\sF}_{\fp})$ the sum of $\rank(\ol{\sF}_{\fp})$ and the Swan conductor of $\ol{\sF}_{\fp}$ over $\spec(\kappa(\fp))$. We put
\begin{equation}\label{dimtot s}
\dimtot_{\fp}(\sF)=-\ord_{s,\fp}(\cc_{\psi_0}(\sF_{\fp}))+\dimtot_{s,{\fp}}(\ol{\sF}_{\fp}).
\end{equation}
We notice that $\ord_{s,\fp}(\cc_{\psi_0}(\sF))$ dose not depend on the choice of $\psi_0$ \eqref{cc ultimate} and the choice of $\p$.
\end{itemize}
We put
\begin{eqnarray}
\varphi_{\eta}(H,U,\sF)&=&\sum_{\fp\in H_{\eta}-U}\dimtot_{\fp}(\sF),\\
\varphi_{s}(H,U,\sF)&=&\sum_{\fp\in P_s(H)}\dimtot_{\fp}(\sF)\label{psi s as}.
\end{eqnarray}

\begin{lemma}[\cite{kato vc} 6.5]\label{invar dimtot}
Let $(H,U,\sF)$ be a stable triple (\ref{triple}), $(S',\eta', s')$ a finite covering of $(S,\eta,s)$. We put $(H',U',\sF')=(H,U,\sF)_{S'}$.
\begin{itemize}
\item[(i)]
For any $\fp\in P_s(H)$ and for the unique $\fp'\in P_s(H)$ above $\fp$, we have \begin{equation*}
\dimtot_{\fp}(\sF)=\dimtot_{\fp'}(\sF').
\end{equation*}
\item[(ii)]
For any $\fp\in H_{\eta}-U$, we have
\begin{equation*}
\dimtot_{\fp}(\sF)=\sum_{\fp'}\dimtot_{\fp'}(\sF'),
\end{equation*}
where $\fp'$ runs over the points above $\fp$.
\end{itemize}
\end{lemma}

\subsection{}
Let $(H,U,\sF)$ be a triple (\ref{triple}). By \ref{be stable}, there exists a finite covering $(S',\eta',s')$ of $(S,\eta,s)$ such that $(H,U,\sF)_{S'}$ is stable. We put
\begin{eqnarray*}
\varphi_{\eta}(H,U,\sF)&=&\varphi_{\eta'}((H,U,\sF)_{S'}),\\
\varphi_{s}(H,U,\sF)&=&\varphi_{s'}((H,U,\sF)_{S'}).
\end{eqnarray*}
By \ref{invar dimtot}, they don't depend on the choice of the covering $(S',\eta',s')$.

\begin{theorem}[Deligne, Kato]\label{deligne kato}
Let $(H,U,\sF)$ be a triple (\ref{triple}), x the closed point of $H$, $u:U\ra H_{\eta}$ the canonical open immersion. Then we have
\begin{equation}\label{d k formula}
\dim_{\L}(\Psi^0_x(u_!\sF))-\dim_{\L}(\Psi^1_x(u_!\sF))=\varphi_{s}(H,U,\sF)-\varphi_{\eta}(H,U,\sF)-2\d(H)\rank(\sF).
\end{equation}
\end{theorem}
\begin{proof}
Indeed, for a stable triple $(H,U,\sF)$ and any $\fp\in P_s(H)$, $\dimtot_{\fp}(\sF)$ is the same as Kato's definition in (\cite{kato scdv} 4.4) by \eqref{general equal}.
\end{proof}

\begin{remark}
The theorem \ref{deligne kato} is proved by Deligne if $\sF$ is unramified at every point of $P_s(H)$ (\cite{lau} 5.1.1). In the general case, Kato proved the theorem with two different definitions of the invariant $\varphi_s(H,U,\sF)$ (\cite{kato vc} 6.7,  \cite{kato scdv} 4.5). T. Saito give another proof with another definition of $\varphi_s(H,U,\sF)$ (\cite{saito tf}) which corresponds to the latter definition of Kato (\cite{kato scdv} 4.5). If $\sF$ is of rank 1, Abbes and Saito gave a definition of $\varphi_s(H,U,\sF)$ (\cite{as ft} A.10) using the refined Swan conductor in their ramification theory \cite{as aml}, which coincides with Kato's latter definition (\cite{kato ch1} remark after 6.8). Here, using Abbes and Saito's ramification theory, we give the definition of $\varphi_s(H,U,\sF)$ for any rank sheaf $\sF$ which is equal to Kato's latter formula (\ref{mainkey}).
\end{remark}


\begin{thebibliography}{20}

\bibitem[AM]{am}
A. Abbes, A. Mokrane,
\emph{Sous-groupes canoniques et cycles \'evanescents $p$-adiques pour les vari\'et\'es ab\'eliennes}
Publ. Math. IHES, \textbf{99} (2004), 117-162.

\bibitem[AS1]{as i}
A. Abbes and T. Saito,
\emph{Ramification of local fields with imperfect residue fields}.
American J. of Math. \textbf{124} (2002), 879-920.

\bibitem[AS2]{as ii}
A. Abbes and T. Saito,
\emph{Ramification of local fields with imperfect residue fields II}.
Documenta Math. Extra Volume Kato (2003), 5-72.

\bibitem[AS3]{as aml}
A. Abbes and T. Saito,
\emph{Analyse micro-locale $\ell$-adique en caract\'eristique $p>0$ le cas d'un trait}.
Publ. RIMS, Kyoto Univ. \textbf{25} (2009), 25-74.

\bibitem[AS4]{as ft}
A. Abbes and T. Saito,
\emph{Local fourier transform and epsilon factors}.
Compositio Math. \textbf{146} (2010), 1507-1551.

\bibitem[AS5]{as rc}
A. Abbes and T. Saito,
\emph{Ramification and cleanliness}.
Tohuku Math. J. Centennial Issue, \textbf{63} No. 4 (2011), 775-853.

\bibitem[BGR]{bgr}
S. Bosch, U. G\"untzer and R. Remmert,
\emph{Non-archimedean analysis}.
A Series of Comprehensive Studies in Mathematics \textbf{261}, Springer-Verlag (1984).

\bibitem[SGA7II]{sga7ii}
P. Deligne and N. Katz,
\emph{Groupes de monodromie en g\'eom\'etrie alg\'ebriques}.
SGA 7 II, LNM \textbf{340}, springer-verlag (1973).

\bibitem[Epp]{epp}
H. Epp,
\emph{Eliminating wild ramification}.
Invent. Math. \textbf{19} (1973), 235-249.

\bibitem[Fu]{fu}
L. Fu,
\emph{Etale cohomology theory}.
Nankai Tracts in Mathematics, Vol. \textbf{13}, world scientific (2011).


\bibitem[EGA IV]{ega4}
A. Grothendieck, J.A. Dieudonn\'e,
\emph{\'El\'ements de G\'eometrie Alg\'ebrique,  IV \'Etude locale des sch\'emas et des morphisms de sh\'emas}.
Publ. Math. IHES, \textbf{20} (1961), \textbf{24} (1965), \textbf{28} (1966), \textbf{32} (1967).

\bibitem[SGA7I]{sga7i}
A. Grothendieck {\it et al.}
\emph{Groupes de monodromie en g\'eom\'etrie alg\'ebriques}.
SGA 7 I, LNM \textbf{288}, springer-verlag (1970).




\bibitem[deJ]{dejong}
A.J. de Jong,
\emph{Crystalline Dieudonn\'e module theory via formal and rigid geometry}.
Publ. Math. IHES, \textbf{82} (1995), 5-96.

\bibitem[Kato1]{kato vc}
K. Kato,
\emph{Vanisihing cycles, ramification of valuations, and class field theory}.
Duke Math. J. Vol.\textbf{55} No.3, (1987), 629-659.

\bibitem[Kato2]{kato scdv}
K. Kato,
\emph{Swan conductors with differential values}.
Advanced Studies in Pure Math. \textbf{12} (1987), 315-342.

\bibitem[Kato3]{kato ch1}
K. Kato,
\emph{Swan conductors for characters of degree one in the imperfect residue field case}.
Contemporary Mathematics, Volume \textbf{83} (1989), 101-131.

\bibitem[Katz]{katz}
N. Katz,
\emph{Gauss sum, Kloosterman sums, and monodromy groups}.
Ann. of Math. Stud. \textbf{116}, Princeton University Press (1988).

\bibitem[Lau1]{lau}
G. Laumon,
\emph{Semi-continuit\'e du conducteur de Swan (d'apr\`es P. Deligne)}.
In the Euler-Poincar\'e characteristic, Ast\'erisque, \textbf{82-83} (1981), 173-219.

\bibitem[Lau2]{lautf}
G. Laumon,
\emph{Transformation de Fourier, constantes d'\'equations fonctionnelles et conjecture de Weil}.
Publ. Math. de l'IHES, tome \textbf{65}, (1987), 131-210.

\bibitem[OW]{ow}
A. Obus, S. Wewers,
\emph{Cyclic extensions and the local lifting problem}.
arXiv:1203.5057, (2012).

\bibitem[Sa1]{saito tf}
T. Saito,
\emph{Trace formula for vanishing cycles of curves},
Math. Ann. \textbf{276} (1987), 311-315.

\bibitem[Sa2]{saito cc}
T. Saito,
\emph{Wild ramification and the characteristic cycle of an $\ell$-adic sheaf}.
J. Inst. Math. Jussieu \textbf{8} (2009), 769-829.

\bibitem[Sa3]{as iii}
T. Saito,
\emph{Ramification of local fields with imperfect residue fields III}.
Math. Ann. Volume \textbf{352}, Issue 3, (2012), 567-580.

\bibitem[Sa4]{wrcb}
T. Saito,
\emph{Wild Ramification and the cotangent bundle}.
arXiv:1301.4632v4, (2013).

\bibitem[Se1]{serre cl}
J.P. Serre,
\emph{ Corps locaux}.
Deuxieme edition, Hermann, (1968).


\bibitem[Se2]{serre gr}
J.P. Serre,
\emph{Linear representations of finite groups}.
Graduate Texts of Mathematics \textbf{42}, Springer-Verlag (1977).

\bibitem[Xiao]{xiao}
L. Xiao,
\emph{On ramification filtrations and $p$-adic differential equations, I: equal characteristic case}.
Algebra and Number Theory, \textbf{4-8} (2010), 969-1027.

\end{thebibliography}
\end{document}